\documentclass{amsart}

\usepackage{latexsym}
\usepackage{amsmath,amsthm,amssymb, amscd,color}
\usepackage{mathrsfs}
\usepackage[all]{xy}

\usepackage{tikz}
\usepackage{tikz-cd}
\usetikzlibrary{snakes}

\usepackage{blkarray} 
\usepackage{framed}

\usepackage{lipsum}

\theoremstyle{plain}
\newtheorem{theorem}{Theorem}[section]
\newtheorem{lemma}[theorem]{Lemma}

\numberwithin{equation}{section}

\theoremstyle{definition}
\newtheorem{definition}[theorem]{Definition}
\newtheorem{example}[theorem]{Example}

\newtheorem{proposition}[theorem]{Proposition}
\newtheorem{remark}[theorem]{Remark}

\theoremstyle{remark}

\newcommand{\CC}{\mathbb{C}}
\newcommand{\QQ}{\mathbb{Q}}
\newcommand{\RR}{\mathbb{R}}
\newcommand{\ZZ}{\mathbb{Z}}
\newcommand{\NN}{\mathbb{N}}
\newcommand{\SR}{\mathcal{SR}}
\newcommand{\CP}{\mathbb{C}P}
\newcommand{\cR}{\mathcal{R}}

\def\skp#1{\vskip#1in\relax}
\def\nd{\noindent}

\newcommand{\coker}{\operatorname{coker}}

\def\blue#1{\textcolor{blue}{#1}}

\def \mbf{\mathbf}
\def \mc{\mathcal}

\def \ms{\mathcal}

\begin{document}

\title[Equivariantly formal toric orbifolds]{Infinite families of equivariantly formal toric orbifolds}

\author[A. Bahri]{Anthony Bahri}
\address{Department of Mathematics, Rider University, NJ, USA}
\email{bahri@rider.edu}
%

\author[S. Sarkar]{Soumen Sarkar}
\address{Department of Mathematics, Indian Institute of Technology Madras,
Chennai, India}
\email{soumensarkar20@gmail.com}

\author[J. Song]{Jongbaek Song}
\address{Department of Mathematical Sciences, KAIST, Daejeon, Republic of Korea}
\email{jongbaek.song@gmail.com}
%


\dedicatory{}

\begin{abstract}
The simplicial wedge construction on simplicial complexes 
and simple polytopes has been used by a variety of authors 
to study toric and related spaces, including non-singular 
toric varieties, toric manifolds, intersections of quadrics and 
more  generally, polyhedral products. In this paper we extend 
the analysis to include toric orbifolds. Our main results yield 
infinite families of toric orbifolds, derived from a given one, 
whose integral cohomology is free of torsion and is 
concentrated in even degrees, a property which might be 
termed \emph{integrally equivariantly formal}. In all cases, 
it is possible to give a description of the cohomology ring 
and to relate it to the cohomology of the original orbifold.
\end{abstract}
\maketitle

\section{Introduction}\label{sec_introduction}
Compact smooth toric varieties $X$, and their topological 
counterparts, toric manifolds\footnote{Though Davis and 
Januszkiewicz \cite{DJ} introduced the name toric manifolds, 
in recent literature they are sometimes called {\em quasitoric 
manifolds\/}.}, have their integral cohomology  torsion free 
and concentrated in even degrees. Consequently,  the action  
of the compact $n$-dimensional torus $T = (S^1)^n$ on $X$, 
allows for a satisfying description of the integral cohomology 
ring arising from the collapsing Serre spectral sequence of 
the canonical fibration
\begin{equation}
X \overset{\iota}{\hookrightarrow} ET \times_{T} X \xrightarrow{\pi}  BT\cong (\CP^\infty)^{n}.
\end{equation}
Franz and Puppe note in \cite[Theorem 1.1]{FP}, that for 
compact smooth toric varieties, or more generally, a finite 
$T$-CW complexes, this is equivalent to the map $\iota^{\ast}$ 
inducing an isomorphism
\begin{equation}\label{btfree}
H^{\ast}\big(ET \times_{T} X\big)\otimes_{H^{\ast}(BT)}\mathbb{Z} \longrightarrow H^{\ast}(X).
\end{equation}

\nd The term {\em integrally equivariantly formal\/} suggests 
itself for $X$, by analogy with the case of rational coefficients, 
see for instance \cite{GKM}. 
In this smooth setting, \eqref{btfree} leads to a description of 
the ring $H^{\ast}(X)$ as a quotient of a Stanley--Reisner ring 
by a linear ideal.

The situation for {\em singular\/} toric spaces is not so nice; 
examples with non-vanishing cohomology in odd degrees 
abound, \cite{Fis}, \cite[Chapter 12]{CLS}. In the case of 
orbifolds however, the results of \cite{BNSS} establish 
tractable sufficient conditions ensuring that the integral 
cohomology is torsion free and concentrated in even degree. 
So, for such spaces an integral equivariant formality holds 
and the ring structure of the cohomology can be identified 
explicitly, in a manner entirely analogous to the smooth case 
described above; the Stanley--Reisner ring is replaced by a ring
of certain {\em piecewise polynomials} or \emph{weighted 
Stanley--Reisner rings}, see \cite[Proposition 2.2]{BFR} and 
\cite[Theorem 5.3]{BSS}.

In this sequel to \cite{BNSS}, we extend the program to a 
class of infinite families of toric orbifolds, derived from a given 
one by a combinatorial construction, known variously  as: 
{\em the simplicial wedge construction\/} \cite{PB}, \cite{Ew}, 
{\em the doubling construction\/} \cite{LDM} and  
{\em the J-construction\/} \cite{BBCG18, BBCG15}.

The construction associates to a sequence of positive 
integers $J = (j_1,j_2,\ldots,j_m)$ and an $(n-1)$-dimensional 
simplicial complex $K$ on $m$ vertices, a new simplicial complex 
$K(J)$ on $d(J) = j_1 + j_2 + \cdots + j_m$ vertices, of dimension 
$d(J)-m+n-1$. Equally well, it associates to a simple polytope 
$Q$ of dimension $n$ with $m$ facets, a new simple polytope 
$Q(J)$ of dimension $d(J)-m+n$ having $d(J)$ facets.
(Notice that $m-n$ = $d(J) - (d(J)-m+n)$.)

As outlined in Section \ref{sec_toric_orb_orb_lens} below, 
a $2n$-dimensional toric orbifold  $X(Q,\lambda)$ is specified by an 
$\mathcal{R}$-characteristic pair $(Q,\lambda)$ where $Q$ is an 
$n$-dimensional simple polytope and $\lambda\colon \mathcal{F}(Q) 
\to \mathbb{Z}^{n}$ is a function from the set of facets 
of $Q$ which satisfies certain conditions. 

Section \ref{sec_J-const} describes the way in which each sequence 
$J$, determines from $X(Q,\lambda)$ a new 
$2\big(d(J)-m+n\big)$-dimensional toric orbifold 
$X_{(J)} \mathrel{\mathop:}= X(Q_{(J)},\lambda_{(J)})$. 
Our goal here is two-fold:\\[-6mm]
\begin{enumerate}
\item To confirm that if $X(Q,\lambda)$ satisfies the sufficiency 
conditions of \cite[Theorem 1.1]{BNSS}, which ensure that the 
integral cohomology is torsion free and concentrated in even 
degree,  then as $J$ varies, the spaces $X(Q_{(J)},\lambda_{(J)})$  
yield an infinite family of similarly integrally equivariantly formal 
orbifolds.
\item To relate the integral cohomology ring of $X_{(J)}$ to that of $X$.
\end{enumerate}

The elementary  theory of toric orbifolds is reviewed in Section 
\ref{sec_toric_orb_orb_lens}, drawing from the expositions to be 
found in \cite{DJ}, \cite{PS} and \cite{BSS}. The emphasis is on 
tracking the finite isotropy groups and the singularities resulting
from the failure of the $\mathcal{R}$--characteristic function 
$\lambda$ to satisfy the regularity condition which ensures the 
smoothness of a toric space.

The orbifold analogue of a CW-complex, called a {\bf q}-CW 
complex,  developed rationally in the toric space setting by 
Poddar and Sarkar \cite{PS} and integrally in \cite{BNSS}, 
is reviewed in Section \ref{sec_bdseq_and_q-cell}. 
The {\bf q}-CW complex structure is connected to the underlying 
combinatorics by the notion of a {\em retraction sequence\/} 
for a simple polytope, introduced in \cite{BSS}. The outcome 
of these observations allows  for an iterated construction of 
the toric orbifold via a sequence of cofibrations which keep 
track of the isotropy, and hence singularities, as they arise. 
A condition involving the divisibility of all the orders of the 
isotropy groups, which emerge from {\em particular\/} retraction 
sequences, proves sufficient to establish the integral 
equivariant formality of the toric orbifold. The cornerstone 
of several of the results presented here,  is the next theorem.
\skp{0.05}
\nd {\bf Theorem \ref{thm_no-torsion}} \cite[Theorem 4.6]{BNSS}  
{\em Let\/} $X := X(Q,\lambda)$ {\em be a toric orbifold.\/} 
{\em If for each prime number $p$ there is a retraction sequence\/} 
$\{(B_j, E_j, b_j)\}_{j=1}^{\ell}$ {\em such that\/} 
gcd$\{p,|G_{E_j}(b_j)|\} = 1$ {\em for all\/} $j$, {\em then\/} 
$H_{\ast}(X; \mathbb{Z})$  {\em has no torsion and 
$H_{\text{odd}}(X; \mathbb{Z})$ is trivial\/}.
\skp{0.05}
In Section \ref{sec_J-const}, various formulations of the simplicial 
wedge construction are introduced on both simplicial complexes
and simple polytopes. This is followed by a description of the 
transition of $\mathcal{R}$-characteristic pairs
$$X(Q,\lambda)  \rightsquigarrow  X(Q_{(J)},\lambda_{(J)})$$
\nd for each sequence $J = (j_1,j_2,\ldots,j_m) \in \mathbb{N}^m$.  
Included here is the verification, in Lemma 
\ref{lem_lambda(J)_satisfies_the_condition}, that if the 
$\mathcal{R}$-characteristic map $\lambda$ satisfies the 
orbifold condition \eqref{def_R-char_fun}, then $\lambda_{(J)}$ 
does too. This is a modification of the argument for the smooth 
case, \cite[Theorem 3.2]{BBCG15}.  The construction of 
$Q_{(J)}$ from $Q$ is an iterative process involving a sequence 
of ``doubling'' operations. Though it is well  known that the result 
is independent of the particular sequence chosen, we include here 
Proposition \ref{prop_(2,1,...1)(2,1,...,1)=(3,1,...,1)} and the details 
of a proof for completeness.

Section \ref{sec_homology_of_X_{(J)}} is devoted to the homology 
groups of the toric orbifold $X_{(J)}:=X(Q_{(J)}, \lambda_{(J)})$; 
the main theorem is the following.
\skp{0.05}
\nd {\bf Theorem \ref{thm_X_satisfies_assump_then_X(J)_eq_formal}} 
{\it Let $X := X(Q,\lambda)$ be a toric orbifold satisfying the 
assumption of Theorem $3.6$. Then, $H_{\ast}(X_{(J)}; \mathbb{Z})$ 
is torsion free and $H_{\text{odd}}(X_{(J)}; \mathbb{Z})$ vanishes 
for arbitrary $J = (j_1,j_2,\ldots,j_m) \in \mathbb{N}^m$}.
\skp{0.05}
The paper concludes with a discussion of the integral cohomology 
ring of the toric orbifolds $X_{(J)}$. When the toric orbifold 
$X(Q,\lambda)$ satisfies the hypothesis of Theorem 
\ref{thm_no-torsion}, then Theorem \ref{thm_cohom_of_X(J)}
characterizes the cohomology of $X_{(J)}$ as 
the quotient of a certain {\em weighted\/} Stanley--Reisner ring
$w\mathcal{SR}[Q_{(J)},\lambda_{(J)}]$ by a linear ideal which 
depends on the $\mathcal{R}$--characteristic map $\lambda$. 
The final result relates the ring 
$w\mathcal{SR}[Q_{(J)},\lambda_{(J)}]$ to the ring $w\mathcal{SR}[Q,\lambda]$. 

\subsection*{Acknowledgments}
This work was supported in part by grants 210386 and 426160 from Simons Foundation. 
The third author has been supported by 
the POSCO Science Fellowship of POSCO TJ Park Foundation.

\section{Toric orbifolds}
\label{sec_toric_orb_orb_lens}
In this section, we review the basic theory of toric orbifolds \cite{DJ, PS} constructed from 
a combinatorial information called an $\cR$-characteristic pair \cite{BSS}.  
Given an $n$-dimensional simple convex polytope $Q$, let 
$V(Q)=\{v_1, \dots, v_\ell\}$ be the set of vertices and 
$\ms{F}(Q)=\{F_1, \dots, F_m\}$ the set consisting of codimension 
1 faces called \emph{facets} of $Q$. 

\begin{definition}
A pair $(Q, \lambda)$ consisting of an $n$-dimensional simple polytope $Q$ and 
a function $\lambda \colon \ms{F}(Q)\to \ZZ^n$,
is called an $\mathcal{R}$-\emph{characteristic pair} if the following condition is satisfied:
\begin{align}\label{def_R-char_fun}
\{\lambda(F_{i_1}), \dots, \lambda(F_{i_k})\}\text{ is a linearly independent set, whenever }
\bigcap_{j=1}^k F_{i_j}\neq \emptyset.
\end{align}
In this case, we call $\lambda$ an $\mathcal{R}$-\emph{characteristic function} on $Q$. 
\end{definition}

An $\mathcal{R}$-characteristic function is often represented by an 
$n\times m$ matrix whose $i$-th column vector is the transpose of $\lambda(F_i)$
for $i=1, \dots,  m$. We call
 this matrix the \emph{characteristic matrix} associated to $\lambda$.

Given an $(n-k)$-dimensional face $E= F_{i_1}\cap \dots \cap F_{i_k}$ of $Q$
for $k\geq 1$, 
let $M(E)$ be the $\ZZ$-submodule of $\ZZ^n$
generated by the set $\{ \lambda(F_{i_1}), \dots, \lambda(F_{i_k})\}$. 
Then, $M(E)$ induces a free $\ZZ$-submodule 
$\left(M(E) \otimes_\ZZ \RR \right) \cap \ZZ^n$ of rank $k$ in $\ZZ^n$. 
Hence, we can define a natural projection 
$$\rho_E \colon \ZZ^n \to \ZZ^{n-k} \cong \ZZ^n/ 
( \left(M(E) \otimes_\ZZ \RR \right) \cap \ZZ^n).$$
Next, set
$$ \mathcal{F}(E):=
\{E \cap F_j \mid F_j\notin \{F_{i_1}, \dots, F_{i_k}\} \text{ and } 
F_j\cap E\neq \emptyset\}.$$
Now, we define a function 
\begin{equation}\label{eq_lambda_E}
\lambda_E \colon \mathcal{F}(E) \to \ZZ^{n-k}
\quad \text{by}\quad  \lambda_E(E\cap F_j) = { prim}((\rho_E\circ \lambda)(F_j)),
\end{equation}
where $ {prim}((\rho_E\circ \lambda)(F_j))$ denotes the primitive vector of 
$(\rho_E\circ \lambda)(F_j).$
One can check that the pair $(E, \lambda_E)$ is an $\mathcal{R}$-characteristic
pair. We also associate $M(E)$ with a $k$-dimensional subtorus $T_E$ of 
standard $n$-dimensional real torus $T^n$
generated by the images of $\lambda(F_{i_1}), \dots, \lambda(F_{i_k})$
under the map $M(E) \hookrightarrow \ZZ^n \xrightarrow{\exp} T^n$.

An $\cR$-characteristic pair $(Q, \lambda)$ 
determines a $2n$-dimensional orbifold
with an action of $n$-dimensional torus $T^n$. 
To be more precise, 
for each point $x\in E$, let  $E(x)$ be the face of $Q$ containing $x$ 
in its interior. Now we consider the following quotient space  
$$X(Q, \lambda):=(T^n \times Q)/\sim,$$
where
\begin{equation}\label{eq_equivalence_rel}
 (t,x)\sim (s,y) ~~ \mbox{if and only if} ~~ x=y ~~\mbox{and} ~~ t^{-1}s\in T_{E(x)}. 
\end{equation}
The space $X(Q, \lambda)$ is equipped with an action of $T^n$ 
given by the multiplication on the first factor. The orbit map 
is induced by the projection onto the second factor,
\begin{equation}\label{eq_orbit_map}
\pi\colon X(Q, \lambda) \to Q
\end{equation}
defined by  $[t, x]_{\sim} \mapsto x.$
A detailed verification that $X(Q,\lambda)$ has an orbifold structure, 
including an explicit description of the orbifold charts, is 
contained in \cite[Section 2]{PS}. The authors also give an  
axiomatic description and show that it agrees with the 
construction above up to equivariant homeomorphism.
The space $X(Q, \lambda)$ is known as a 
\emph{toric orbifold}\footnote{Davis 
and Januszkiewicz \cite{DJ} first called them toric orbifolds.
In recent literature, they are sometimes called \emph{quasitoric orbifolds}.}.

\begin{remark}
If the collection $\{\lambda(F_{i_1}), \dots, \lambda(F_{i_k})\}$ in 
\eqref{def_R-char_fun} is unimodular whenever $F_{i_1}\cap \dots\cap  F_{i_k}\neq \emptyset$,
then $\lambda$ is called a \emph{characteristic function} and the 
resulting space $X(Q, \lambda)$ is a smooth \emph{toric manifold},
see \cite[Section 1]{DJ}.
\end{remark}

\begin{remark}\label{rmk_preimage_of_face}
An $\cR$-characteristic pair $(Q, \lambda)$
induces another $\cR$-characteristic pair $(E, \lambda_E)$ for each 
face $E$ of $Q$, which defines another toric orbifold 
$X(E, \lambda_E)$. On the other hand, $\pi^{-1}(E)$ is an invariant
suborbifold of $X(Q, \lambda)$ with respect to the action of $T^n$
on $X(Q, \lambda)$. Indeed, $\pi^{-1}(E)$ is the suborbifold fixed by the subtorus $T_E$ of 
$T^n$. The residual torus $T^n/T_E \cong T^{\dim E}$ acts on the 
suborbifold $X(E, \lambda_E)$ and $\pi^{-1}(E)$.
It is shown in \cite[Section 2.3]{PS} that 
$X(E, \lambda_E)$ is equivariantly homeomorphic to $\pi^{-1}(E)$.
In particular, when $E$ is a $1$-dimensional face of $Q$, then 
$\pi^{-1}(E)$ is homeomorphic to $S^2$. 
\end{remark}

%
%

%

For simplicity, we summarize notation that we introduced above as follows. 
\begin{enumerate}
\item $(Q, \lambda)$ is an $\mathcal{R}$-characteristic pair 
with $\dim Q=n$ and 
$X(Q, \lambda)$ is the associated toric orbifold of dimension $2n$.  
\item $E=F_{i_1} \cap \dots \cap F_{i_k}$ is an $(n-k)$-dimensional 
face of $Q$, for $k \geq 1$.  
\item $\lambda_E \colon \mathcal{F}(E) \to \ZZ^{n-k}$ is an 
$\mathcal{R}$-characteristic function induced from $(Q, \lambda)$. 
\item $X(E, \lambda_E)$ is the toric orbifold associated to the $\cR$-characteristic pair 
$(E, \lambda_E)$. 
\item\label{item_(5)} For a vertex $v$ in $E$, we denote by $\Lambda_{E, v} \colon \ZZ^{n-k} \to \ZZ^{n-k}$ 
the linear map given by the square matrix 
$$\Lambda_{E,v}=\left[ \begin{array}{c|c|c} 
\lambda_E(E\cap F_{s_1})^t& \dots& \lambda_E(E\cap F_{s_{n-k}})^t
\end{array} \right],$$
where $v=\bigcap_{a=1}^{n-k} (E\cap F_{s_a})$.
In particular, when $E=Q$ and $v=\bigcap_{a=1}^n F_{s_a}$,  
$$\Lambda_{v}:=\Lambda_{Q, v}=\left[ \begin{array}{c|c|c} \lambda(F_{s_1})^t& \dots& \lambda( F_{s_{n}})^t \end{array} \right].$$
\item $G_E(v) := \ker (\exp \Lambda_{E,v} \colon T^{n-k} \twoheadrightarrow T^{n-k}).$ 
\label{page_local_group}
\end{enumerate}

Notice that the notation $G_E(v)$ is used for $\coker \Lambda_{E,v}$ in 
\cite[Section 4]{BSS}. However, the following commutative diagram and 
snake lemma, see for instance \cite{Ati}, shows that those two finite groups are
isomorphic. 
\begin{equation}\label{eq_kerexp=coker}
\begin{tikzcd}[row sep=small]
&&0\dar &\ker (\exp \Lambda_{E, v}) \dar & \\
0 \rar &\ZZ^{n-k} \arrow{dd}{\Lambda_{E, v}} \rar& 
\ZZ^{n-k} \otimes_\ZZ \RR \rar \arrow{dd}{\cong} & 
T^{n-k}\rar \arrow{dd}{\exp \Lambda_{E, v}} &0 \\
&&&&\\
0 \rar &\ZZ^{n-k} \rar\dar & \ZZ^{n-k} \otimes_\ZZ \RR \rar \dar & T^{n-k}\rar &0 \\
& \coker \Lambda_{E, v} &0  & & 
\end{tikzcd}
\end{equation}

\begin{remark}\label{rmk_order_of_local_group}
One can see from \eqref{eq_kerexp=coker} that the order $|G_E(v)|$ 
of the finite group $G_E(v)$ is exactly $|\det \Lambda_{E, v}|$. 
\end{remark}

\begin{example}\label{ex_induced_char_ftn}
Let $Q$ be a 3-dimensional prism and $F_1, \dots, F_5$ its facets illustrated in 
Figure \ref{fig_prism_and_char_ftn}. 
\begin{figure}
\begin{tikzpicture}[scale=0.7]
\draw[dashed] (0,0)--(1,-1)--(3,1)--cycle; 
\draw (0,2)--(1,1)--(3,3)--cycle; 
\draw (0,0)--(1,-1)--(1,1)--(0,2)--cycle; 
\draw (1,1)--(1,-1)--(3,1)--(3,3)--cycle; 
\draw[dashed] (0,0)--(3,1)--(3,3)--(0,2)--cycle; 

\node at (1/2, 1/2) {$F_1$};
\node at (2.3, 1.2) {$F_2$};
\node at (4/3, 2) {$F_4$};
\node[right] at (4,2) {$F_3$};
\node[right] at (4,0) {$F_5$};

\draw[dotted, thick, ->] (4,2.3) to [out=120, in=90] (2.5,2);
\draw[dotted, thick, ->] (4,0) to [out=240, in=330] (4/3,0);

\draw plot [mark=*, mark size=2] coordinates{(1,1)};
\draw plot [mark=*, mark size=2] coordinates{(0,2)};
\node[left] at (0,2) {$v_1$};
\node[right] at (1,1) {$v_2$};
\end{tikzpicture}
\caption{}
\label{fig_prism_and_char_ftn}
\end{figure}
We assign $\mathcal{R}$-characteristic vectors by 
$\lambda(F_1)=(1,0,0),~\lambda(F_2)=(0,1,0),~\lambda(F_3)=(0,0,1),~
\lambda(F_4)=(1,2,4)$ and $\lambda(F_5)=(-1,-1,-1)$. 
Observe that 
\begin{align*}
G_Q(v_1) &=\{(t_1, t_2, t_3) \in T^3 \mid t_1t_3=t_3^2=t_2t_3^4=1\}\\
&= \langle (1,1,1), (-1, 1, -1) \rangle \cong \ZZ/2\ZZ, 
\end{align*}
Similarly we have 
\begin{align*}
G_Q(v_2) &=\{ (t_1, t_2, t_3) \in T^3 \mid t_1t_3=t_2t_3^2=t_3^4=1\}  \\
&= \langle (1,1,1), (-i, -1, i), (-1, 1, -1), (i, -1, -i)\rangle \cong \ZZ/4\ZZ,\\
G_Q(v) & = \langle (1,1,1) \rangle \cong 1,~ \text{ for } v\in V(Q) \setminus \{v_1, v_2\}.
\end{align*}
Choosing the face $F_4$, we consider $\lambda(F_4)=(1,2,4),~(0,1,0)$ and $(0,0,1)$ as a basis 
of $\ZZ^3$. Then, the induced $\cR$-characteristic function 
$$\lambda_{F_4} \colon \{F_4 \cap F_1, F_4\cap F_2, F_4\cap F_3\} \to \ZZ^2$$
on $F_4$ is given by 
\begin{align*}
\lambda_{F_4} (F_4 \cap F_1) = (-1,-2), ~~ 
\lambda_{F_4} (F_4 \cap F_2) = (1,0) ~~ \mbox{and} ~~ 
\lambda_{F_4} (F_4 \cap F_3) = (0,1).
\end{align*}
To be more precise, since $\lambda(F_1)=(1,0,0)=(1,2,4)-2(0,1,0)-4(0,0,1),$
we have $(\rho_{F_4}\circ \lambda)(F_1)=(-2,-4)$ whose primitive vector gives 
$\lambda_{F_4}(F_4 \cap F_1)$. A similar computation gives other two induced
$\cR$-characteristic vectors. Moreover, we have 
\begin{align*}
G_{F_4}(v_1)&=\{ (t_1, t_2)\in T^2 \mid t_1^{-1}=t_1^{-1}t_2=1\}
=\left< (1,1)\right>\cong 1,\\
G_{F_4}(v_2)&=\{ (t_1, t_2)\in T^2 \mid t_1^{-1}t_2=t_2^{-2}=1\}
=\left< (1,1), (-1, -1)\right>\cong \ZZ/2\ZZ.
\end{align*}
Finally, $G_{F_4}(v)$=1, where $v=F_2\cap F_3\cap F_4$ is considered as a vertex of $F_4$. 
\end{example}

We finish this section by the following proposition which we shall use in 
Section \ref{sec_homology_of_X_{(J)}}. One can see this property in 
Example \ref{ex_induced_char_ftn}. 

\begin{proposition}\cite[Proposition 4.3]{BSS}\label{prop_sub_local_group}
Let $E$ and $E'$ be two faces of $Q$ containing a vertex $v$ such that 
$E$ is a face of $E'$. Then, $|G_E(v)|$ divides $|G_{E'}(v)|$. 
\end{proposition}

\section{Building sequences and homology of toric orbifolds}
\label{sec_bdseq_and_q-cell}
The goal of this section is to discuss integral homology of toric orbifolds.
After Poddar--Sarkar \cite{PS} developed the notion of $\mathbf{q}$-CW complex, 
the authors of \cite{BNSS} introduced a \emph{building sequence} 
which enables us to detect $p$-torsion freeness of an orbifold 
having $\mathbf{q}$-CW complex structure for each prime number $p$.
In case of toric orbifolds, one can derive a building sequence via a \emph{retraction 
sequence} of a simple polytope introduced in \cite{BSS}. In this subsection, we
review the definition of a \emph{building sequence} and a 
\emph{retraction sequence} and study their relation. 

\subsection{Building sequence}

Let $\bar{D}^n$ be a closed $n$-dimensional disc and $G$ a finite group 
acting linearly on $\bar{D}^n$. We call the quotient $\bar{D}^n/G$ a 
\emph{$\mathbf{q}$-disc}. Now, a $\mbf{q}$-CW complex is defined inductively
in the similar manner as a usual CW complex, but we use $\mbf{q}$-discs instead 
of $\bar{D}^n$. 
To be more precise, we start with a discrete set $X_0$ of 
$0$-dimensional $\mbf{q}$-cells. Next, assuming $X_{i-1}$ is defined, 
we define 
$$X_i := X_{i-1} \cup_{\{f_\alpha\}} \{\bar{D}^{i}/G_\alpha\},$$
where $f_\alpha \colon \partial (\bar{D}^{i}/G_\alpha) \to X_{i-1}$ is the
\emph{attaching map} for a $\mbf{q}$-cell $\bar{D}^{i}/G_\alpha$ for 
finitely many $\alpha$.  

A \emph{building sequence} for a $\mbf{q}$-CW complex $X$ is a sequence 
$\{Y_j\}_{j=1}^\ell$ of $\mbf{q}$-CW subcomplexes of $X$ such that 
\begin{equation}\label{eq_attaching_cell_in_building_seq}
Y_j\setminus Y_{j-1} \cong D^{k_j} / G_j
\end{equation}
for some $k_j$-dimensional open $\mbf{q}$-disc $D^{k_j}/G_j$. We denote by 
$0_j$ the image of  
origin of \eqref{eq_attaching_cell_in_building_seq} and call it the \emph{free special 
point} in $Y_j$. 
When we need to emphasize the free special point, we denote the 
building sequence by $\{(Y_j, 0_j)\}_{j=1}^\ell$, see \cite[Section 2]{BNSS} for 
more details. 

\begin{example}
Let $Y_1$ be a point and $Y_2= Y_1 \cup _{f_1} \bar{D}^1$  a circle obtained 
by the canonical  attaching map $f_1 \colon S^0 \to \{pt\}$. Consider two
$\mbf{q}$-cells $\bar{D}^2/\ZZ_p$ and  $\bar{D}^2/\ZZ_q$, where we regard
$\ZZ_p$ as a finite group generated by $p$-th root of unity and it acts on
$\bar{D}^2\subset \CC$ by the multiplication. Similarly we consider 
$\ZZ_q$-action on $\bar{D}^2$. Now we define 
$$Y_3: =Y_2 \cup_{f_2}  \bar{D}^2/\ZZ_p \text{ and } 
Y_4: =Y_3 \cup_{f_3}  \bar{D}^2/\ZZ_p,$$ 
where 
$f_2 \colon  S^1/\ZZ_p \to Y_1\cong S^1$, respectively 
$f_3 \colon S^1/\ZZ_q \to Y_1 \subset Y_2 $,
defined by $f_2([e^{2\pi i x}]) = e^{2\pi i px}$, respectively $f_3([e^{2\pi i y}])=e^{2\pi i qy}$, for 
$0 \leq x,y \leq 1$. Then, the resulting space $Y_4$ is a $2$-dimensional orbifold sphere 
of football type $\CP^1_{p,q}$, see Figure \ref{fig_football}. 

\begin{figure}[h!]
\begin{tikzpicture}[scale=0.8]
\draw[fill] (-1.5,0) circle (0.05); 
\draw[fill] (0,0) circle (0.05); 
\node at (-1.5,-1.5) {\footnotesize$Y_1$};
\node at (1,-1.5) {\footnotesize$Y_2=Y_1 \cup_{f_1} \bar{D}^1$};
\begin{scope}[yscale=0.7]
\draw (0,0)  [out=90, in=90] to (2,0) [out=270, in=270] to (0,0);
\end{scope}

\begin{scope}[xshift=110, yscale=0.7]
\draw[fill] (0,0) circle (0.05); 
\draw[dashed] (0,0)  [out=90, in=90] to (2,0) [out=270, in=270] to (0,0);
\draw[fill=yellow, opacity=0.5] 
(1,2) [out=315, in=90] to (2,0) [out=270, in=270] to (0,0) [out=90, in=225] to (1,2);
\draw(1,2) [out=315, in=90] to (2,0) [out=270, in=270] to (0,0) [out=90, in=225] to (1,2);
\node at (1,-2.5) {\footnotesize$Y_3=Y_2 \cup_{f_2} \bar{D}^2/\ZZ_p$};
\end{scope}

\begin{scope}[xshift=220, yscale=0.7]
\draw[fill] (0,0) circle (0.05); 
\draw[dashed] (0,0)  [out=90, in=90] to (2,0) [out=270, in=270] to (0,0);
\draw[fill=blue, opacity=0.4, dashed] 
(1,-2) [out=60, in=270] to (2,0) [out=90, in=90] to (0,0) [out=270, in=120] to (1,-2);
\draw
(0,0) [out=270, in=120] to (1,-2) [out=60, in=270] to (2,0) ;
\draw[dashed] (0,0) [out=90, in=90] to (2,0);
\draw[fill=yellow, opacity=0.5] 
(1,2) [out=315, in=90] to (2,0) [out=270, in=270] to (0,0) [out=90, in=225] to (1,2);
\draw (1,2) [out=315, in=90] to (2,0) [out=270, in=270] to (0,0) [out=90, in=225] to (1,2);
\node at (1,-2.5) {\footnotesize$Y_4=Y_3 \cup_{f_3} \bar{D}^2/\ZZ_q$};
\end{scope}

\end{tikzpicture}
\caption{A building sequence of $\CP^1_{p,q}$.}
\label{fig_football}
\end{figure}
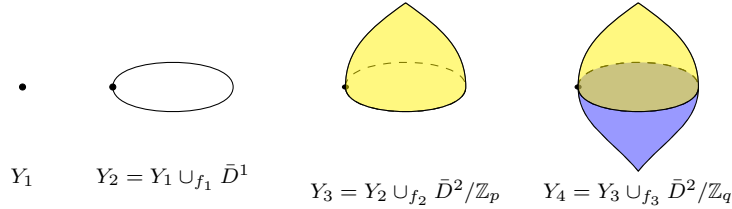
\end{example}

\begin{remark}\label{rmk_trivial_q-cell}
A $\mbf{q}$-disc $\bar{D}^k/G$ is homeomorphic to 
the usual disc $\bar{D}^k$ if $k=0,1$ or $2$. Indeed, 
the football-type orbifold sphere $\CP^1_{p,q}$ is 
homeomorphic to $S^2$ for arbitrary positive integers $p$ and $q$. 
\end{remark}

\subsection{Retraction sequence}
Given an $n$-dimensional simple polytope $Q$, as a polytopal 
complex (see \cite[Definition 5.1]{Zie}), we define a sequence of 
polytopal subcomplexes of $Q$ which will determine a building 
sequence of a toric orbifold whose orbit space is $Q$. We
refer to \cite[Section 2]{BSS} for more details
about retraction sequences of a simple polytope. 

Let $B$ 
be a connected polytopal subcomplex of $Q$. A vertex $v$ of $B$ 
is called a \emph{free vertex} if $v$ has a neighborhood homeomorphic to 
$\RR^N_{\geq 0}$ for some $N\in \{1, \dots, \dim B\}$ as manifolds with corners. In this case, there 
exists a unique maximal face $E$ of $B$ containing $v$ as a vertex. We write $FV(B)$ 
as the set of all free vertices of $B$. 
For instance, every vertex $v$  of $Q$ is a free vertex and $Q$ itself is the 
unique (non-proper) face of $Q$ containing $v$. In particular, $FV(Q)=V(Q)$. 
However, for a polytopal subcomplex $B$ of $Q$, $FV(B)$ is a subset of $V(B)$
in general. 

A retraction sequence $\{(B_j, E_j, b_j)\}_{j=1}^\ell$ for $Q$ is a sequence of 
triples consists of a polytopal subcomplex $B_j$ of $Q$, a free vertex $b_j$ 
of $B_j$ and the unique face $E_j$ of $B_j$ containing a free vertex $b_j$,
which is defined inductively such that 
$$B_{j+1}=\bigcup \{E\mid E \text{ is a face of }B_j \text{ with } b_j\notin E\}.$$
The next term $(B_{j+1}, E_{j+1}, b_{j+1})$ is given by the choice of a free 
vertex $b_{j+1}$ of $B_{j+1}$ and the unique face $E_{j+1}$ 
determined by the edges of $B_{j+1}$ containing $b_{j+1}$. 
Finally, the sequence ends up with $(B_\ell, E_\ell, b_\ell)=(b_\ell, b_\ell, b_\ell)$
for some vertex $b_\ell$ of $Q$. 

We note that every simple polytope has at least one retraction sequence. 
It can be shown by realizing a simple polytope $Q$ as a convex subset in 
$\RR^n$ and taking a linear function $f\colon \RR^n \to \RR$ 
such that $f(v_i)\neq f(v_j)$ whenever two vertices $v_i$ and $v_j$ of $Q$  
are distinct. We refer to \cite[Proposition 2.3]{BSS} for the completeness of arguments.

\begin{example}\label{ex_n-gon}
Let $Q$ be an $\ell$-gon and $v_1, \dots, v_\ell$ its vertices with counterclockwise 
order. Let $F_i$ be the $1$-dimensional face connecting $v_i$ and $v_{i+1}$ for 
$i=1, \dots, \ell-1$ and $F_{\ell}$ the $1$-dimensional face with vertices $v_{\ell}$ and 
$v_1$. Then, one canonical choice of a retraction sequence 
$\{(B_j, E_j, b_j)\}_{j=1}^{\ell}$
is given by 
$$(B_j, E_j, b_j)=\begin{cases} (Q, Q, v_1) & \text{if } j=1;\\
( F_j\cup \dots \cup F_{\ell-1}, F_j, v_j)& \text{if } j=2, \dots, \ell-1;\\
(v_\ell, v_\ell, v_\ell) & \text{if } j=\ell,
\end{cases}$$
as in  Figure \ref{fig_ret_of_polygon} for $5$-gon. 
\end{example}
\begin{figure}
\begin{tikzpicture}[scale=0.8]
\draw[fill=blue!30] (18:1cm) \foreach \x in {90, 90+72, 90+144, 90+216, 18}	
{ -- (\x: 1cm)} ;
\foreach \x/\l/\p in 
{18/{\blue{\small$v_1$}}/right,
90/{\small$v_2$}/above,
90+72/{\small$v_3$}/left,
90+144/{\small$v_4$}/left,
90+216/{\small$v_5$}/right}
\node (\x:1cm) [circle, draw, fill, inner sep = 1pt, label = {\p:\l}] at (\x:1cm) {};

\begin{scope}[xshift=100]
\draw[very thick, blue]  (90:1cm)--(90+72:1cm);
\draw (90:1cm) \foreach \x in {90+72, 90+144, 90+216} { -- (\x: 1cm)} ;
\foreach \x/\l/\p in 
{90/{\blue{\small$v_2$}}/above,
90+72/{\small$v_3$}/left,
90+144/{\small$v_4$}/left,
90+216/{\small$v_5$}/right}
\node (\x:1cm) [circle, draw, fill, inner sep = 1pt, label = {\p:\l}] at (\x:1cm) {};
\end{scope}

\begin{scope}[xshift=200]
\draw[very thick, blue]  (90+144:1cm)--(90+72:1cm);
\draw (90+72:1cm) \foreach \x in {90+144, 90+216} { -- (\x: 1cm)} ;
\foreach \x/\l/\p in 
{
90+72/{\blue{\small$v_3$}}/left,
90+144/{\small$v_4$}/left,
90+216/{\small$v_5$}/right}
\node (\x:1cm) [circle, draw, fill, inner sep = 1pt, label = {\p:\l}] at (\x:1cm) {};
\end{scope}

\begin{scope}[xshift=300]
\draw[very thick, blue]  (90+144:1cm)--(90+216:1cm);
\draw (90+144:1cm) \foreach \x in {90+144, 90+216} { -- (\x: 1cm)} ;
\foreach \x/\l/\p in 
{
90+144/{\blue {\small$v_4$}}/left,
90+216/{\small$v_5$}/right}
\node (\x:1cm) [circle, draw, fill, inner sep = 1pt, label = {\p:\l}] at (\x:1cm) {};
\end{scope}

\begin{scope}[xshift=350]
\foreach \x/\l/\p in 
{
90+216/{\small$v_5$}/right}
\node (\x:1cm) [circle, draw, fill, inner sep = 1pt, label = {\p:\l}] at (\x:1cm) {};
\end{scope}
\end{tikzpicture}
\caption{A retraction sequence for $5$-gon.}
\label{fig_ret_of_polygon}
\end{figure}
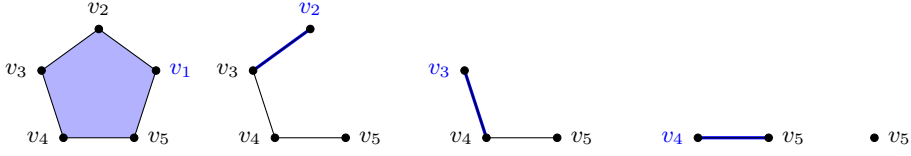

The following proposition shows a relation between a retraction sequence of a 
simple polytope $Q$ and a building sequence of the associated toric orbifold 
$X(Q, \lambda)$. 

\begin{proposition}\cite[Proposition 4.4]{BNSS}\label{prop_ret_build}
Let $(Q, \lambda)$ be an $\mathcal{R}$-characteristic pair and $\pi \colon 
X(Q, \lambda) \to Q$ be the orbit map.  
Then, a retraction sequence $\{(B_j, E_j, b_j)\}_{j=1}^\ell$ for $Q$ induces a building
sequence $\{(Y_{j}, 0_{j})\}_{j=1}^\ell$ as follows:
\begin{itemize}
\item $Y_1=0_i=\pi^{-1}(b_\ell)$, 
\item $Y_j=\bigcup_{s=\ell-j+1}^\ell \pi^{-1}(U_s)$, where $U_s$ is the open subset of 
$E_s$ obtained by deleting faces of $E_s$ which does not contain $b_s$, for $j=2, \dots, \ell$. 
\end{itemize}
In particular, $Y_\ell= X(Q, \lambda)$ and $0_j=\pi^{-1}(b_j)$ for all $j=1, \dots, \ell$. 
\end{proposition}

\begin{remark}
Though a building sequence induced from a retraction 
sequence may require the image of an attaching
map for a q-cell $\overline{D}^k/G$ 
to be in a higher dimensional skeleton, a $G$-equivariant 
triangulation \cite[Theorem 3.6]{Il} of $\overline{D}^k/G$ together with 
an application of the cellular approximation theorem
allows for a deformation into the appropriate dimension.
\end{remark}

Next two theorems can be obtained by combining  
\cite[Theorem 1.1, 1.2]{BNSS} and Proposition \ref{prop_ret_build},  
which describe a sufficient condition for ($p$-)torsion freeness and vanishing odd degree 
(co)homology of a toric orbifold.

\begin{theorem}\label{thm_no_p-torsion}
Let $X:=X(Q, \lambda)$ be a toric orbifold and $p$ a prime number.
If there is a retraction $\{(B_j, E_j, b_j)\}_{j=1}^{\ell}$ such that
$\gcd\{p, |G_{E_j}(b_j)|\}=1$ for all $j$, then $H_{\ast}(X; \ZZ)$ 
has no $p$-torsion and $H_{odd}(X; \ZZ_p)$ is trivial. 
 \end{theorem}

\begin{theorem}\label{thm_no-torsion}
Let $X(Q, \lambda)$ be a toric orbifold. If for each prime $p$, 
there is a retraction $\{(B_j, E_j, b_j)\}_{j=1}^{\ell}$ such that
$\gcd\{p, |G_{E_j}(b_j)|\}=1$ for all $j$, then $H_{\ast}(X; \ZZ)$ 
has no torsion and $H_{odd}(X; \ZZ)$ is trivial. 
 \end{theorem}

Next two examples are applications of Theorem \ref{thm_no-torsion} 
to toric orbifolds over a polygon and a simplex, respectively. Though 
these results are well known from the literature, for example
\cite{Fis, Jor, KMZ} and \cite{Ka}, 
the same conclusions follow from the results above. 

\begin{example}
Let $Q$ be an $\ell$-gon and $F_1, \dots, F_\ell$ facets of $Q$ as in Example 
\ref{ex_n-gon}. Let $\lambda(F_i):=(a_i, b_i)\in \ZZ^2$ be characteristic 
vectors for $i=1, \dots, \ell$. Hence, we have 
$$|G_Q(v_1)|=\left| \det \begin{bmatrix} a_\ell & a_{1} \\ b_\ell & b_{1} \end{bmatrix}\right|
\quad \text{and}  \quad 
|G_Q(v_i)|=\left| \det \begin{bmatrix} a_{i-1} & a_{i} \\ b_{i-1} & b_{i} \end{bmatrix}\right|
\text{ for } i=2, \dots, \ell.$$
We refer to Remark \ref{rmk_order_of_local_group}. 
Moreover, since the facets are $1$-dimensional,  $G_{F_i}(v_i)$ is trivial for each 
facet $F_i$. Indeed, one can easily check that the induced $\mathcal{R}$-characteristic 
function 
$$\lambda_{F_i} \colon \{v_i, v_{i+1}\} \to \ZZ$$ 
can be always defined by 
$\lambda_{F_i}(v_i)=\lambda_{F_i}(v_{i+1})=\pm 1$, see 
Remark \ref{rmk_preimage_of_face} and Remark \ref{rmk_trivial_q-cell}. 
Now, we assume that 
$$\gcd\{|G_Q(v_1)|, \dots,  |G_Q(v_{\ell})| \} =q.$$
Then, one can always choose a retraction sequence $\{(B_j, E_j, b_j)\}_{j=1}^{\ell}$
of $Q$ such that $\gcd\{p, |G_{E_j}(b_j)|\}=1$ unless $p$ is a 
factor of $q$. Hence, we conclude 
that if $H_{\ast}(X(Q, \lambda); \ZZ)$ has a non-trivial torsion part, then it must have
a $p$-torsion for some $p$ dividing $q$. 
In particular, if $q=1$, then $H_{\ast}(X(Q, \lambda); \ZZ)$ is torsion free. 
\end{example}

\begin{example}\label{ex_wCP}
Consider an $n$-simplex $\Delta^n$ and let $\chi:=(\chi_1, \dots, \chi_{n+1})\in \NN^{n+1}$.
An $\cR$-characteristic function 
$$\lambda \colon \mathcal{F}(\Delta^n):=\{F_1, \dots, F_{n+1}\} \to \ZZ^n$$ 
satisfying 
$\sum_{i=1}^{n+1} \chi_i \lambda(F_i)=\mathbf{0}$ and  
${\rm span}_\ZZ \{\lambda(F_1), \dots, \lambda(F_{n+1})\}=\ZZ^n.$
It is well-known that this $\cR$-characteristic pair $(\Delta^n, \lambda)$ 
defines a weighted projective space $\CP^n_{\chi}$, 
see  \cite[Example 3.1.17]{CLS} or \cite[Section 2.2]{Ful}. 
Without loss of generality, we may assume that 
$\gcd\{\chi_1, \dots, \chi_{n+1}\}=1$. Indeed, for some 
$k\in \NN$, two vectors 
$(\chi_1, \dots, \chi_{n+1})$ and $(k\chi_1, \dots, k\chi_{n+1})$  
yield homeomorphic weighted projective spaces, see for example 
\cite[Theorem 1.1]{BFNR}. 
Note that $\chi_i$ is the order of singularity at 
$[0, \dots, 0,\underset{i\text{-th}}{1},0,\dots 0]\in \CP^n_{\chi}$, which is same as 
$|G_{\Delta^n}(v_i)|$ defined from the $\mathcal{R}$-characteristic pair
$(\Delta^n, \lambda)$, where 
$v_i=F_1\cap \cdots \cap F_{i_1} \cap F_{i+1} \cap \cdots F_{n+1}$. 
Now, following the proof of \cite[Proposition 4.5]{BNSS}, 
one can always 
find retraction sequences $\{(B_j, E_j, b_j)\}_{j=1}^{n+1}$ which 
satisfies the assumption of Theorem \ref{thm_no-torsion}. Hence, 
we conclude that the (co)homology of any weighted projective space is 
torsion free and concentrated in even degrees. 
\end{example}

\section{$J$-construction of toric orbifolds}\label{sec_J-const}

For \emph{toric manifolds},  
the authors of \cite{BBCG15} use a construction introduced in \cite{PB} 
to construct new toric manifolds from a given one.
Indeed, this was done 
by producing a new \emph{characteristic pair} from the original 
one in a canonical way. The process for making a new polytope from the 
given one is called \emph{simplicial wedge construction}. 
Moreover, by a successive procedure, 
a countably infinite family of new toric manifolds can 
arise from the original manifold. In \cite{BBCG15}, where 
the construction was analyzed in the 
context of polyhedral products, the process is described efficiently 
by using a vector 
$J=(j_1, \dots, j_m) \in \NN^m$, where 
$m$ is the number of facets in the original polytope. 
To be more precise, given a positive integral vector 
$J=(j_1, \dots, j_m) \in \NN^m$, one can obtain a new 
toric manifold $M(J)$ from the original toric manifold 
$M$. We refer this procedure as the \textit{$J$-construction} 
 and apply it to toric orbifolds. 
As an example, we shall investigate the class of spaces 
that can be produced from a weighted projective spaces. 

\subsection{The simplicial wedge construction}
Let $K$ be a simplicial complex with vertex set $V(K)=\{w_1, \dots, w_m\}$. 
We call a subset $\sigma\subset V(K)$ \emph{a non-face of} $K$  
if $\sigma$ is not a simplex in $K$. 
A non-face $\sigma$ is called \emph{minimal} if every proper subset 
of $\sigma$ is a simplex in $K$.
Then, the combinatorial type of $K$ is determined by the set 
of minimal non-faces of $K$.
	
For arbitrary positive integral vector $J=(j_1, \dots, j_m)\in \NN^m$,  
a new simplicial complex $K_{(J)}$ is defined on 
$V(K_{(J)}) = \{w_{11}, \ldots, w_{1j_1}, \ldots, w_{1m}, \ldots, w_{mj_m} \}$ 
with minimal non-faces of the form $V_{i_1} \cup \cdots \cup V_{i_k}$, 
where $V_i = \{w_{i1}, \ldots, w_{ij_i}\}$, whenever $\{w_{i_1}, \ldots, w_{i_k}\}$ 
is a minimal non-face of $K$.
	
In the special case $J=(1,\dots, 1,\underset{\substack{\uparrow \\ i-\text{th}}}{2},1,\dots, 1)$,  
we denote $K_{(w_i)} := K_{(J)}$, and refer to it as \emph{the simplicial wedge construction} of 
$K$ on $w_i \in V$, see \cite{PB}.  The following representation is a useful 
combinatorial description of the simplicial wedge construction. 
\begin{equation}\label{eq_wedge_of_K}
K_{(w_i)}=
\big[\{w_{i1}, w_{i2}\} \ast {\rm link}_{K}\{w_i\}\big] \cup 
\big[\{\{w_{i1}\}, \{w_{i2}\}\}\ast (K \setminus \{w_i\})\big],
\end{equation}
where $\ast$ denotes the join of two simplicial complexes
and we identify $w_{r1}$ with $w_{r}$ for $r\neq i$. 

In this paper, we focus on the case when $K$ is dual to the boundary $\partial Q$ of 
a simple polytope $Q$, which we denote by $K_Q$ and refer to as the 
\emph{nerve complex} of $Q$, see for instance \cite[Section 2.2]{BP-book}. 
Notice that the vertex set 
$V(K_Q)=\{w_1, \dots, w_m\}$ bijectively corresponds to the set of facets 
$\mathcal{F}(Q)=\{F_1, \dots, F_m\}$.

\begin{example} \label{ex_simplicial_wedge}
	\begin{enumerate}
	\item Let $K=K_{\Delta^n}$ be the nerve complex of  an $n$-simplex $\Delta^n$
	and $V(K)=\{w_1,\dots,w_{n+1}\}$ its vertex set. Then, there exists 
	only one minimal non-face $\sigma=\{w_1,\dots ,w_{n+1}\}$.  
	For an arbitrary $J=(j_1, \dots, j_{n+1}) \in \NN^{n+1}$,  $K_{(J)}$
	is the simplicial complex on the vertex set 
	$$\{w_{11}, \dots, w_{1j_1}, \dots, w_{n+1,1}, \dots , w_{n+1, j_{n+1}}\}$$ 
	with the unique minimal non-face 
	$\{w_{11}, \dots, w_{1,j_1}, \dots, w_{n+1,1}, \dots , w_{n+1,j_{n+1}}\}$. 
	Hence, we get $K_{(J)}=K_{\Delta^{d(J)}}$, where 
	$d(J):= {\sum_{i=1}^{n+1} j_i}$. 
	Figure \ref{Fig_simp_wedge_of_simplex} describes the case 
	when $n=2$, $J=(1,1,2)$ and the decomposition by \eqref{eq_wedge_of_K}. 
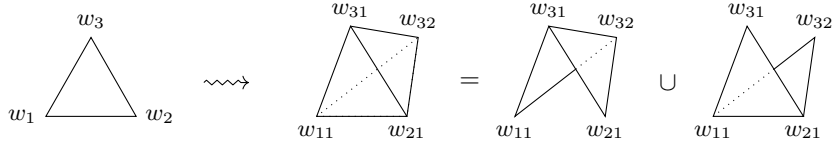
\begin{figure}[h]
	\begin{tikzpicture}[scale=0.3]
            \draw (0,9.5)--(4,9.5)--(2,13)--cycle;
            \node[left] at (0,9.5) {\small{$w_1$}};
            \node[right] at (4,9.5) {\small{$w_2$}};
            \node[above] at (2,13) {\small{$w_3$}};
            \draw[line join=round, decorate, decoration=
            {zigzag, segment length=4, 
            amplitude=.9,post=lineto, post length=2pt}, ->] (7,11)--(9,11);
            \draw   (12,9.5)--(16,9.5)--(13.5,13.5)--cycle;
            \draw [dotted] (12,9.5)--(16,9.5)--(16.5,13)--cycle;
            \draw (13.5,13.5)--(16.5,13)--(16,9.5);
            \node[below] at (12,9.5) {\small$w_{11}$};
            \node[below] at (16,9.5) {\small$w_{21}$};
            \node[above] at (13.5,13.5) {\small$w_{31}$};
            \node[above] at (16.5, 13) {\small$w_{32}$};
\begin{scope}[xshift=250]
\node at (10, 11) {$=$};
\draw  (13.5,13.5)--(16,9.5)--(16.5, 13)--cycle;
\draw [dotted] (288/107+12, 224/107+9.5)--(16.5, 13);
\draw (288/107+12, 224/107+9.5)--(12,9.5)--(13.5,13.5);
\node[below] at (12,9.5) {\small$w_{11}$};
\node[below] at (16,9.5) {\small$w_{21}$};
\node[above] at (13.5,13.5) {\small$w_{31}$};
\node[above] at (16.5, 13) {\small$w_{32}$};
\end{scope}

\begin{scope}[xshift=500]
\node at (10, 11) {$\cup$};
\draw  (12,9.5)--(16,9.5)--(13.5,13.5)--cycle;
\draw [dotted] (12,9.5)--(288/107+12, 224/107+9.5);
\draw (288/107+12, 224/107+9.5)--(16.5,13)--(16, 9.5);
\node[below] at (12,9.5) {\small$w_{11}$};
\node[below] at (16,9.5) {\small$w_{21}$};
\node[above] at (13.5,13.5) {\small$w_{31}$};
\node[above] at (16.5, 13) {\small$w_{32}$};
\end{scope}
	 \end{tikzpicture}
\caption{$(K_{\Delta^2})_{(w_3)}=K_{\Delta^3}$.}
\label{Fig_simp_wedge_of_simplex}
\end{figure}	
\item Consider next the nerve complex 
$$K:=K_{\Delta^{n_1} \times  \Delta^{n_2}}= 
K_{\Delta^{n_1}} \ast K_{\Delta^{n_2}}.$$
of  product of two simplices $\Delta^{n_1}$ and $\Delta^{n_2}$. 
Suppose that  $\{v_1, \dots, v_{n_1+1}\}$ and
$\{w_1, \dots, w_{n_2+1}\}$ are vertex sets of $K_{\Delta^m}$ 
and $K_{\Delta^n}$, respectively. 
Then, $K$ is a simplicial complex on the vertex set 
$\{v_1, \dots, v_{n_1+1},w_1, \dots, w_{n_2+1}\}$
with two minimal non-faces 
$\{v_1, \dots, v_{n_1+1}\}$ and $\{w_1, \dots, w_{n_2+1}\}$. 
Taking $J=(2,1, \dots, 1)\in \NN^{n_1+n_2+2}$, we obtain 
$$K_{(J)}=K_{(v_1)}=K_{\Delta^{n_1+1}}\ast K_{\Delta^{n_2}}.$$
See Figure \ref{fig_J-fication_of_simp.cpx} for the case when $n_1=n_2=1$. 
\begin{figure}[h]
        \begin{tikzpicture}[scale=0.3]
            \draw (0,2)--(2,4)--(4,2)--(2,0)--cycle;
            \node[left] at (0,2) {\small{$v_1$}};
            \node[above] at (2,4) {\small{$w_1$}};
            \node[right] at (4,2) {\small{$v_2$}};
            \node[below] at (2,0) {\small{$w_2$}};
            \draw[line join=round, decorate, decoration={zigzag, segment length=4, 
            			amplitude=.9,post=lineto, post length=2pt}, ->] (6.5,2)--(8.5,2);
\begin{scope}[xshift=-20]
            \draw (14.5,5)--(17,2)--(14.5,0)--(12.5,1.5)--(12,3)--cycle;
            \draw (17,2)--(12.5,1.5)--(14.5,5);
            \draw[dotted] (17,2)--(12,3)--(14.5,0);
            \node[above] at (14.5,5) {\small{$w_1$}};
            \node[above] at (17.3,2) {\small{$v_2$}};
            \node[below] at (14.5,0) {\small{$w_2$}};
            \node[left] at (12.5,1.5) {\small{$v_{11}$}};
            \node[left] at (12,3) {\small{$v_{12}$}};
\end{scope}
\begin{scope}[xshift=270]
\node at (9, 2) {$=$};
            \draw (12.5,1.5)--(14.5,5)--(12,3)--(12.5,1.5)--(14.5, 0);
            \draw[dotted] (12,3)--(18/59+12.5, 63/118+3/2);
            \draw (18/59+12.5, 63/118+3/2)--(14.5,0);
            \node[above] at (14.5,5) {\small{$w_1$}};
            \node[below] at (14.5,0) {\small{$w_2$}};
            \node[left] at (12.5,1.5) {\small{$v_{11}$}};
            \node[left] at (12,3) {\small{$v_{12}$}};
\end{scope}
\begin{scope}[xshift=500]
\node at (8.5, 2) {$\cup$};
            \draw[dotted] (18/59+12.5, 63/118+3/2)--(14.5,0);
            \draw[dotted] (95/78+12, -19/78+3)--(17,2);
            \draw (12.5, 1.5)--(14.5,5)--(17,2)--cycle;
            \draw (12.5,1.5)--(14.5, 0)--(17,2);
            \draw (14.5, 5)--(12,3)--(18/59+12.5, 63/118+3/2);
            \draw (12,3)--(95/78+12, -19/78+3);
            \node[above] at (14.5,5) {\small{$w_1$}};
            \node[below] at (14.5,0) {\small{$w_2$}};
            \node[left] at (12.5,1.5) {\small{$v_{11}$}};
            \node[left] at (12,3) {\small{$v_{12}$}};
\node[above] at (17.3,2) {\small{$v_2$}};
\end{scope}

        \end{tikzpicture}
        \caption{$(K_{\Delta^1} \ast K_{\Delta^1})_{(w_1)}$.}
        \label{fig_J-fication_of_simp.cpx}
     \end{figure}
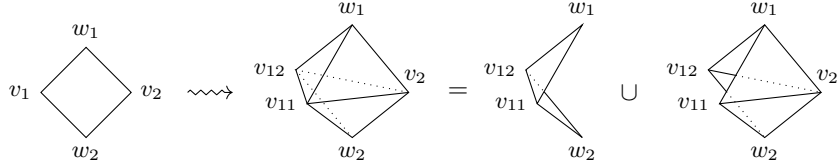
\item In general, if $K_1$ and $K_2$ are simplicial complexes on 
$\{v_1, \dots, v_{m_1}\}$ and $\{ w_1, \dots, w_{m_2}\}$, respectively, then
\begin{equation}\label{eq_wedge_of_join}
(K_1 \ast K_2)_{(v_i)}={K_1}_{(v_i)}\ast K_2 \text{ and } 
(K_1 \ast K_2)_{(w_{j})}=K_1\ast {K_2}_{(w_j)}
\end{equation}
for $i\in \{1, \dots, m_1\}$ and $j\in \{1, \dots, m_2\}$. Indeed, one can see these 
relations by comparing minimal non-faces. 
\end{enumerate}
\end{example}

\subsection{The polytopal wedge construction}\label{subsec_polytopal_wedge}
According to \cite[page 582]{PB}, if $K$ is a dual to the boundary of a simple polytope $Q$, 
then $K_{(w_i)}$ is again a simplicial 
complex which is dual to the boundary of a simple polytope. 
Notice that, for an arbitrary $J\in \NN^m$, $K_{(J)}$ can be 
obtained by the iterated procedure of simplicial wedge 
construction. Hence, we can see that $K_{(J)}$ is also dual to 
the boundary of a simple polytope which we denote by $Q_{(J)}$.

In particular, when $J=J':=(1,\dots,1,2,1,\dots,1)\in \NN^m$, $Q_{(J')}$ 
is homeomorphic to 
\begin{equation}\label{eq_polytopal_wedge}
Q_{(J')}:=(Q\times I)/_{\sim_{F_i}}, \text{ where  } (x, t)\sim_{F_i} (y, 0) \text{ if } x=y \in F_i
\end{equation}
as manifolds with corners. Indeed, $(Q\times I)/_{\sim_{F_i}}$ has the following 
facets
$$\big\{Q^+, Q^-\big \} \cup 
\big\{(F\times I)/_{\sim_{F_i}} ~\mid~ F\in \mathcal{F}(Q)\setminus\{F_i\}\big\}$$
where
$Q^+:=Q\times \{1\}$ and $Q^-:=Q\times \{0\}$ 
intersect  at $(F_i \times I)/_{\sim_{F_i}} \cong F_i$ and each of 
$Q^+$ and $Q^-$ intersects
all other facets $\big\{(F\times I)/_{\sim_{F_i}} \mid F\in \mathcal{F}(Q)\setminus\{F_i\}\big\}$. 
Notice that this observation is exactly the dual 
representation of \eqref{eq_wedge_of_K}
given by associating $Q^-$, $Q^+$, $F\times I /_{F_i}$ for 
$F\in \mathcal{F}(Q)\setminus \{ F_i\}$
with $w_{i1}$, $w_{i2}$ and $w\in V(K)\setminus \{w_i\}$, respectively, 
where $w$ is the dual of $F$. 
We call $Q_{(J')}$  the \emph{polytopal wedge construction} of $Q$ 
with respect to $F_i\in \mathcal{F}(Q)$. We may also denote 
$Q_{(J')}$ by $Q_{(F_i)}$ to emphasize the chosen facet $F_i$. 
See Figure \ref{fig_polytopal_wedge} 
for the example of the polytopal wedge construction of $5$-gon. 

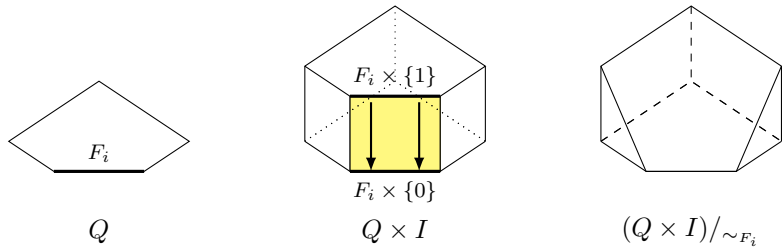
\begin{figure}[h]
\begin{tikzpicture}[scale=0.4]
\begin{scope}
		\draw (0,2)--(3,4)--(6,2)--(4.5,1)--(1.5,1)--cycle;
		\node at (3,-1) {$Q$};
		\draw[very thick] (4.5,1)--(1.5,1);
		\node[above] at (3,1) {\footnotesize$F_i$};
\end{scope}	

\begin{scope}[xshift=280]
\draw[fill=yellow, opacity=0.5] (4.5,1)--(1.5,1)--(1.5,3.5)--(4.5, 3.5)--cycle;
		\draw (6,2)--(4.5,1)--(1.5,1)--(0,2);
		\draw[dotted] (0,2)--(3,4)--(6,2);
		\draw (0,4.5)--(3,6.5)--(6,4.5)--(4.5,3.5)--(1.5,3.5)--cycle;
		\draw[dotted] (0,2)--(3,4)--(6,2);
		\draw[dotted] (3,6.5)--(3,4);
		\draw (0,4.5)--(0,2);
		\draw (1.5,3.5)--(1.5,1);
		\draw (4.5,1)--(4.5,3.5);
		\draw (6,2)--(6,4.5);
\begin{scope}[>=latex]
\draw[thick,->] (3.8,3.3)--(3.8,1);
\draw[thick, ->] (2.2,3.3)--(2.2,1);
\end{scope}
\draw[very thick] (4.5,1)--(1.5,1);
\draw[very thick] (4.5,3.5)--(1.5,3.5);
\node[above] at (3, 3.5) {\footnotesize$F_i\times \{1\}$};
\node[below] at (3, 1) {\footnotesize$F_i\times \{0\}$};
		\node at (3,-1) {$Q\times I$};
\end{scope}	

\begin{scope}[xshift=560]
		\draw (6,2)--(4.5,1)--(1.5,1)--(0,2);
		\draw[dashed] (0,2)--(3,4)--(6,2);
		\draw (0,4.5)--(3,6.5)--(6,4.5)--(4.5,1)--(1.5,1)--cycle;
		\draw[dashed] (0,2)--(3,4)--(6,2);
		\draw[dashed] (3,6.5)--(3,4);
		\draw (0,4.5)--(0,2);
		\draw (6,2)--(6,4.5);
		\node at (3,-1) {$(Q\times I)/_{\sim_{F_i}}$};
\end{scope}	
	\end{tikzpicture}
\caption{A polytopal wedge construction of $5$-gon.}
\label{fig_polytopal_wedge}
\end{figure}

The following example is dual to Example \ref{ex_simplicial_wedge}. 
\begin{example}\label{ex_polytopal_wedge}
\begin{enumerate}
\item $\Delta^n_{(F)}=\Delta^{n+1}$, where $F$ is any facet of $\Delta^n$. 
\item $(\Delta^{n_1} \times \Delta^{n_2})_{(E\times \Delta^{n_1})}= 
\Delta^{n_1+1} \times \Delta^{n_2}$, 
where $E$ is a facet of $\Delta^{n_1}$, see Figure \ref{fig_J-fication_of_simp.cpx} for the case 
when $n_1=n_2=1$. 
\item In general, given two simple polytopes $P$ and $Q$, any facet of $P\times Q$ 
is of the form $E\times Q$ for some facet $E$ of $P$, or $P\times F$ for some 
facet $F$ of $Q$. Then, the relation 
$K_{P\times Q} = K_P \ast  K_Q$ 
together with \eqref{eq_wedge_of_join}
leads us to the following:
\begin{equation*}\label{eq_wedge_of_product_of_polytopes}
(P\times Q)_{(E\times Q)}= P_{(E)}\times Q \quad \text{and} \quad 
(P\times Q)_{(P\times F)}= P\times Q_{(F)}.
\end{equation*}

\end{enumerate}
\end{example}

\subsection{A new $\cR$-characteristic function}
Let $(Q, \lambda)$ be an $\mathcal{R}$-characteristic pair and  $K$ 
the simplicial complex dual to $\partial Q$ as above. 
As in previous sections, $m$ and $n$ denote
the number of facets of $Q$ and the dimension of $Q$, respectively. 
Given a vector $J=(j_1, \dots, j_m)\in \NN^m$, we define a matrix 
$\Lambda_{(J)}$ of size 
$(d(J)-m+n)) \times d(J) $ as follows, where 
$d(J):= \sum_{i=1}^m j_i$ ;
\begin{align}\label{eq_lambda(J)_matrix}
\Lambda_{(J)}= \tiny{\left[ \begin{array}{ccc|ccc|ccc|ccc|ccccc} 
        &&&&&&&&&&&&-1&& \\                           
        & I_{j_1-1}& & &  & & &  & & &  & &  \vdots &&  &  \\                           
        &&&&&&&&&&&&-1&&& \\ \hline                           
        &&&&&&&&&&&&&-1&&\\                          
        &&&&I_{j_2-1}&&&&&&&&&\vdots& &\\                           
        &&&&&&&&&&&&&-1&&\\ \hline                           
        &&&&&&&&&&&&&&&\\                           
        &&&&&&&\ddots&&&&&&&\vdots& \\                           
        &&&&&&&&&&&&&&&\\ \hline                           
        &&&&&&&&&&&&&&&-1\\                           
        &&&&&&&&&&I_{j_m-1}&&&&&\vdots\\                           
        &&&&&&&&&&&&&&&-1\\ \hline                           
        &&&&&&&&&&&&&&&\\                           
        &&&&&&&&&&&&&\Lambda& & \\                           
        &&&&&&&&&&&&&&&                     
\end{array}\right]},
\end{align}    
where all the columns of the matrix are indexed respectively by
\begin{equation}\label{eq_column_index_of_lambda(J)}
\{ w_{12}, \ldots , w_{1{j_1}}, w_{22}, \ldots , w_{2{j_2}}, \ldots , w_{m2}, 
\ldots , w_{m{j_m}}, w_{11}, \ldots , w_{m1} \},
\end{equation} 
all of the entries in the empty spaces are zero and $\Lambda$ is the original 
characteristic matrix associated to $\lambda$. Notice that 
the indexing in \eqref{eq_column_index_of_lambda(J)} bijectively corresponds to 
the vertex set $V(K_{(J)})$. 

Let $F_{ik}$ be the facet of $Q_{(J)}$ dual to the vertex $w_{ik}$, 
where $i\in\{1, \dots, m\}$ and $k\in \{1, \dots, j_i\}$. 
Then, the matrix $\Lambda_{(J)}$ defines a function 
\begin{equation}\label{eq_lambda(J)}
\lambda_{(J)} \colon \mathcal{F}(Q_{(J)}) \to \ZZ^{d(J)-m+n},
\end{equation}
by assigning to the facet $F_{ik}$ the transpose of the column 
vector of $\Lambda_{(J)}$ indexed by $w_{ik}$. 

If $\lambda$ satisfies Davis and Januszkiewicz's regularity condition 
$(\ast)$, \cite[p.423]{DJ}, 
then so does $\lambda_{(J)}$ for all $J\in \NN^m$, see \cite[Theorem 3.2]{BBCG15}.
The same proof goes through by replacing condition $(\ast)$ 
$$\det \left[\begin{array}{c|c|c} 
\lambda(F_{i_1})^t & \cdots &\lambda(F_{i_{n}})^t \end{array}\right]=\pm 1$$
with the orbifold condition \eqref{def_R-char_fun}
$$\det \left[\begin{array}{c|c|c} 
\lambda(F_{i_1})^t & \cdots &\lambda(F_{i_{n}})^t \end{array}\right]\neq 0,$$
to give the next lemma.
\begin{lemma}\label{lem_lambda(J)_satisfies_the_condition}
Let $(Q, \lambda)$ be an $\mathcal{R}$-characteristic pair. Then, 
for arbitrary $J\in \NN^m$, the function \eqref{eq_lambda(J)} satisfies the 
orbifold condition \eqref{def_R-char_fun}. 
\end{lemma}
\noindent Hence from  $(Q, \lambda)$, one can obtain an infinite family of toric orbifolds
\begin{equation*}
X_{(J)}:=X(Q_{(J)}, \lambda_{(J)})
\end{equation*}
for arbitrary $J\in \NN^m$. 

\begin{example}\label{ex_J-cons_of_w_proj_sp}
Let $Q$ be an $n$-simplex and consider the $\mathcal{R}$-characteristic 
function from Example \ref{ex_wCP}. 
As in Example \ref{ex_simplicial_wedge}-(1), for an arbitrary vector 
$J=(j_1, \dots, j_{n+1})\in \NN^{n+1}$, 
$$\Delta^n_{(J)}=\Delta^{d(J)-1}.$$ 
Then the function $\lambda_{(J)}$, defined by \eqref{eq_lambda(J)}, satisfies
the equation 
$$\sum_{k=1}^{j_1} \chi_1 \lambda_{(J)}(F_{1k}) + \dots + 
\sum_{k=1}^{j_{n+1}}\chi_n\lambda_{(J)}(F_{n+1,k}) =\mathbf{0},$$
and one can show that the new characteristic vectors span the whole 
lattice $\ZZ^{d(J)-1}$, since the original $\mathcal{R}$-characteristic 
vectors in Example \ref{ex_wCP} span $\ZZ^n$.   
Hence, we conclude that $(\CP^n_{\chi})_{(J)}$ is the weighted 
projective space $\CP^{d(J)-1}_{\chi_{(J)}}$, 
where 
$$\chi_{(J)}=
( \underbrace{\chi_1,\ldots,\chi_1}_{j_1}, \ldots , 
\underbrace{\chi_k,\ldots,\chi_k}_{j_k}, \ldots,  
\underbrace{\chi_{n+1},\ldots,\chi_{n+1}}_{j_{n+1}})\in \NN^{d(J)}.$$
\end{example}

In the special case $J' :=  (1,\dots,1,2,1,\dots,1)$, the matrix of \eqref{eq_lambda(J)_matrix} 
takes a particularly simple form, 
\begin{equation}\label{eq_char_matrix_J'}
\Lambda_{(J')}=\begin{bmatrix}
1 & 0 & \cdots &  -1  & \cdots & 0\\
0 &  &  &    &  & \\
\vdots & \lambda(F_1) & \cdots&   \lambda(F_i) & 
 \cdots &\lambda(F_m)\\
0 &  &  &      & & 
\end{bmatrix}.
\end{equation}
Hence, the characteristic function $\lambda_{(J')} \colon \mathcal{F}(Q_{(J')}) \to \ZZ^{n+1}$ 
is defined by 
\begin{align*}
\lambda_{(J')}(Q^+)&=(1, 0, \dots, 0)^t,\\
\lambda_{(J')}(Q^-)&=(-1, \lambda(F_i))^t \text{ and }\\
\lambda_{(J')}(F_{s1})&=(-1, \lambda(F_s))^t,
\text{ for } s\in \{1, \dots, i-1, i+1, \dots, m\}.
\end{align*}

\begin{remark}
The two induced $\cR$-characteristic functions ${\lambda_{(J')}}_{Q^+}$ and  
${\lambda_{(J')}}_{Q^-}$ coincide with $\lambda$. 
This implies that $X_{(2, 1,\dots, 1)}$ has two copies of original orbifold 
$X$ as suborbifolds defined in Section \ref{sec_toric_orb_orb_lens}. 
Ewald \cite{Ew} called $X_{(2,1,\dots,1)}$ the \emph{canonical extension} of $X$, 
when $X$ is a toric variety. We refer to \cite{BBCG18,BBCG15}, \cite{CP, CP2nd}
and \cite{Ew} for more topological and geometrical observations about the 
wedge operation on toric manifolds. 
\end{remark}

The following proposition confirms that $X_{(J)}$ for arbitrary 
$J=(j_1, \dots, j_m) \in \NN^m$ can be constructed from 
iterated wedge operations as mentioned in \cite[Remark 3.1]{BBCG15}.
Here, we give an explicit proof. 

\begin{proposition}\label{prop_(2,1,...1)(2,1,...,1)=(3,1,...,1)}
Let $X$ be the toric orbifold associated to an $\cR$-characteristic pair $(Q, \lambda)$. 
Then, two toric orbifolds $X_{(3,1,\dots,1)}$ and $Y_{(2,1,\dots,1)}$ where 
$Y=X_{(2,1,\dots, 1)}$ are homeomorphic. 
\end{proposition}
\begin{proof}
Let $\{w_1, \dots, w_m\}$ be the vertices of $K:=K_Q$. 
Then, one can see from the definition of $K_{(J)}$ or \eqref{eq_wedge_of_K} that 
both $K_{(3,1,\dots,1)}$ and $(K_{(2,1,\dots,1)})_{(2,1,\dots,1)}$ have the same number of 
vertices and have the same minimal non faces. 
Therefore, the simple polytopes 
$Q_{(3,1,\dots,1)}$ and $(Q_{(2,1\dots,1)})_{(2,1,\dots,1)}$
determined by 
$K_{(3,1,\dots,1)}$ and $(K_{(2,1,\dots,1)})_{(2,1,\dots,1)}$, respectively, 
are homeomorphic as manifolds with corners. 
%
%

The two characteristic matrices $\Lambda_{(3,1,\dots, 1)}$ and 
$(\Lambda_{(2,1,\dots,1)})_{(2,1,\dots,1)}$ differ by an element 
$$\begin{pmatrix}
1& -1 & 0 & \cdots & 0\\
0&  1 & 0 & \cdots & 0\\
 \vdots  &     & \ddots & &\vdots  \\
 \vdots & & & 1& 0 \\
  0 &\cdots&\cdots &0&1
\end{pmatrix} \in SL_{n+2}(\ZZ),$$
 which induces an automorphism 
$$\phi \colon T^{n+2} \to T^{n+2}$$
given by $\phi(t_1, \dots, t_{n+2})= (t_1t_2^{-1}, t_2, \dots, t_{n+2})$. 
Finally, the map  
$$\phi \times id \colon  T^{n+2} \times Q_{(3,1,\dots, 1)} \to 
T^{n+2} \times (Q_{(2,1\dots,1)})_{(2,1,\dots,1)}$$
induces a homeomorphism  from $X_{(3,1,\dots, 1)}$ to 
$(X_{(2,1,\dots,1)})_{(2,1,\dots,1)}$. 
\end{proof}

\section{Homology of $X_{(J)}$}\label{sec_homology_of_X_{(J)}}
In this section, we shall see that the homology of $X_{(J)}$ depends on 
$J$ and the homology of $X$.
First we compare retraction sequences for the two polytopes $Q$ and $Q_{(J)}$. 
It suffices to consider the case when $J'=(1, \ldots, 1, 2, 1, \ldots, 1)$,
because $Q_{(J)}$ can be constructed by the iterations of 
the polytopal wedge construction, as in Subsection \ref{subsec_polytopal_wedge}. 
We assume that the entry $2$ appears in $i$-th coordinate of $J'$, hence 
it corresponds to the $i$-th facet $F_i$ of $Q$. 

Let $V(Q)=\{v_1, \dots, v_\ell\}$ be the vertices of $Q$ and 
 $V(F_i):=\{v_{i_1}, \ldots, v_{i_k}\}$ the vertices of $F_i$. 
Now, the vertices of $Q_{(J')}$ are identified  as 
\begin{equation}\label{eq_vertices_of_Q_tilde}
V(Q_{(J')})=\left\{ v^+, v^- \mid v\in  V(Q)\setminus V(F_i) \right\} \cup 
\big\{v_{i_1}^-, \ldots, v_{i_k}^-\big\}, 
\end{equation}
where we write $v^+:=v\times \{1\}$ and $v^-:=v\times \{0\}$ for notational convenience. 
%
%
%
%
%

Now, we introduce the following two lemmas about the finite group 
$G_{Q_{(J')}}(u)$, as defined in Section \ref{sec_toric_orb_orb_lens}, 
associated to each vertex $u$ of $Q_{(J')}$.

\begin{lemma}\label{lem_loc_goup_isom_for_v_0}
For each  vertex $v_{i_r} \in V(F_i)$,  $r\in \{1, \dots,k\}$, 
the finite group $G_{Q_{(J')}}(v_{i_r}^-)$ is 
isomorphic to $G_Q(v_{i_r}).$ 
\end{lemma}
\begin{proof}
Given a vertex $v_{i_r} \in V(F_i)$, 
assume that $v_{i_r}=F_i \cap F_{s_1} \cap \cdots \cap F_{s_{n-1}}$ 
for some $\{s_1, \ldots, s_{n-1}\} \subset \{1, \ldots, m\}$.
Then, we have 
\begin{equation}\label{eq_v_in_F_i_intersection_of_facets}
v_{i_r}^- = Q^+ \cap Q^- \cap \bigcap_{a=1}^{n-1} {F_{s_a}}_{(F_i\cap F_{s_a})},
\end{equation}
where ${F_{s_a}}_{(F_i\cap F_{s_a})}$ is the polytopal wedge construction, as in 
Subsection \ref{subsec_polytopal_wedge}, by regarding the facet $F_{s_a}$ 
as a simple polytope, and $F_i\cap F_{s_a}$ as a facet of $F_{s_a}$ for $a=1, \dots, n-1$. 

Restricting to the facets which meet at $v_{i_r}^-$, we get from \eqref{eq_char_matrix_J'}
\begin{align*}
{\Lambda_{(J')}}_{v_{i_r}^-}
&=\left[\begin{array}{c|c|c|c|c}
\lambda_{(J')}(Q^{+})^t & \lambda_{(J')}(Q^{-})^t & 
\lambda_{(J')}({F_{s_1}}_{(F_i\cap F_{s_1})})^t& 
\cdots &
\lambda_{(J')}({F_{s_{n-1}}}_{(F_i\cap F_{s_{n-1}})})^t
\end{array}\right]\\
&=\left[ \begin{array}{ccccc}
1 & -1 & 0 & \cdots &0 \\
0&&&&\\
\vdots &\lambda(F_{i})^t&\lambda(F_{s_1})^t&\cdots&\lambda(F_{s_{n-1}})^t\\
0 &&&& \end{array}\right],
\end{align*}
which induces an endomorphism 
$\exp ({\Lambda_{(J')}}_{v_{i_r}^-}) \colon T^{n+1} \to T^{n+1}$. 
According to the definition in Page \pageref{page_local_group} of 
Section \ref{sec_toric_orb_orb_lens}, we have 
\begin{align*}
G_{Q_{(J')}}(v_{i_r}^-) &= 
\ker \left( \exp ({\Lambda_{(J')}}_{v_{i_r}^-}) \colon T^{n+1} 
\twoheadrightarrow T^{n+1}\right) \\ 
&=\{ (t_1, \dots, t_{n+1}) \in T^{n+1} \mid t_1=t_2,~ (t_2, \dots, t_{n+1}) \in G_Q(v_{i_r})\}\\
&\cong G_Q(v_{i_r}). 
\end{align*}
\end{proof}

Next, we consider vertices away from $F_i$. 
\begin{lemma}\label{lem_loc_group_isom_for_v_pm}
Let  $v$ be a vertex in $V(Q) \setminus V(F_i)$. Then, 
the finite groups $G_{Q_{(J')}}(v^+)$  
and  $G_{Q_{(J')}}(v^-)$ are isomorphic to $G_Q(v).$ 
\end{lemma}
\begin{proof}
Suppose $v=F_{s_1}\cap \dots \cap F_{s_n}$ with $i\notin \{s_1, \dots, s_n\}$. 
Then, we have 
\begin{equation}\label{eq_vertex_in_Q(J)_not_in_F_i}
v^-=Q^-\cap (F_{s_1}\times I)\cap \cdots \cap (F_{s_n}\times I).
\end{equation}
The $\mathcal{R}$-characteristic function $\lambda_{(J')}$ yields the 
square matrix   
\begin{align*}
{\Lambda_{(J')}}_{v^-} &=
 \left[\begin{array}{c|c|c|c}
{\lambda_{(J')}}(Q^-)^t & 
{\lambda_{(J')}}(F_{s_1}\times I)^t & 
\cdots &
{\lambda_{(J')}}(F_{s_n}\times I)^t
\end{array}\right]\\
&= \left[\begin{array}{c|c|c|c}
-1& 0 & \cdots & 0 \\
\lambda(F_i)^t & \lambda(F_{s_1})^t & \cdots &\lambda(F_{s_n})^t
\end{array}\right]\\
&= \left[ \begin{array}{c|ccc}
-1 & 0 & \cdots & 0\\ \hline
&&& \\
\lambda(F_i)^t & & \Lambda_v &\\
&&&
\end{array}\right].
\end{align*}
Hence the kernel of the endomorphism 
$\exp({\Lambda_{(J')}}_{v^-})\colon T^{n+1} \to T^{n+1}$ 
of tori induced by 
${\Lambda_{(J')}}_{v^-}$ is 
$$\{ (t_1, \dots, t_{n+1}) \in T^{n+1} \mid t_1=1,~ (t_2, \dots, t_{n+1}) \in G_Q(v)\}$$
which is isomorphic to $G_Q(v)$. 
Similarly one can show that $G_{Q_{(J')}}(v^+)$ is also isomorphic to 
$G_Q(v)$. 
\end{proof}

The observations above allow us now to adapt the 
hypothesis of Theorem \ref{thm_no_p-torsion} to the 
$\mathcal{R}$-characteristic pair $(Q_{(J)}, \lambda_{(J)})$. 

\begin{lemma}\label{lem_ret_simj}
Given an $\mathcal{R}$-characteristic pair $(Q, \lambda)$ and a prime number $p$, 
suppose that  there exist a retraction sequence  $\{(B_r, E_r, b_r)\}_{r=1}^{\ell}$
of $Q$ satisfying the condition $\gcd\{p, |G_{E_r}(b_r)|\}=1$ for $r=1, \ldots, \ell$.  Then 
for an arbitrary $J=(j_1, \dots, j_m)\in \NN^m$, there exists a retraction sequence 
$\{(B^{\prime}_s, E^{\prime}_s, b_s^{\prime})\}_{s=1}^{\ell'}$ for $Q_{(J)}$
which satisfies  $\gcd\{p, |G_{E^{\prime}_s}(b_s^{\prime})|\}=1$ for $s=1, \ldots, \ell'$, 
where $\ell':=|V(Q_{(J)})|$. 
\end{lemma}
\begin{proof}
It is enough to consider the case when $J=J'=(1,\dots, 1,2,1,\dots, 1)$ 
from the discussion in Section \ref{sec_J-const} and the opening 
remark of Section \ref{sec_homology_of_X_{(J)}}. 
We assume the entry $2$ appears in $i$-th coordinate of $J$ and let 
$b_{\beta_1} \dots, b_{\beta_k}$ be the vertices of the facet $F_i$ of $Q$, 
where $\beta_1 < \cdots < \beta_k$. 

Given a prime number $p$ and a retraction sequence 
$\{(B_r, E_r, b_r)\}_{r=1}^{\ell}$ for $Q$ such that 
$\gcd\{p, |G_{E_r}(b_r)|\}=1$ for $r=1, \ldots, \ell$, 
we now construct a retraction sequence 
$\{(B^{\prime}_s, E^{\prime}_s, b_s^{\prime})\}_{s=1}^{2\ell-k}$
for $Q_{(J')}$ satisfying the hypothesis.
To accomplish this, consider the following sequence of vertices of $Q_{(J')}$:
\begin{align}\label{eq_Q(J)_seq_of_free_vertices}
\begin{split}
b_1^+ \to b_1^-\to  \cdots \to b_{\beta_1-1}^+\to  b_{\beta_1-1}^-\to  b_{\beta_1}^-\to 
b_{\beta_1+1}^+\to  b_{\beta_1+1}^-\to\cdots\\
\cdots \to  b_{\beta_k-1}^+\to  b_{\beta_k-1}^-\to  
b_{\beta_k}^{-}\to  b_{\beta_k+1}^+\to  b_{\beta_k+1}^-\to \cdots\to b_\ell^+\to  b_\ell^-. 
\end{split}
\end{align}
The sequence \eqref{eq_Q(J)_seq_of_free_vertices} 
begins with $b_1^-$ if $b_1$ is a vertex of $F_i$, i.e., $\beta_1=1$.
Now, we construct a retraction sequence satisfying the hypothesis 
using the sequence \eqref{eq_Q(J)_seq_of_free_vertices} above. 

\begin{enumerate}
\item[(Case 1)]
We first assume that $b_1$ is a vertex of $F_i$. Then, we take 
$b_1'=b_1^-$ as in \eqref{eq_Q(J)_seq_of_free_vertices} and set  
$(B_1', E_1', b_1')= (Q_{(J')}, Q_{(J')}, b_1^-)$. Next, 
the choice of $b_{1}^{-}$ as a free vertex of $B_1'$ gives
\begin{align*}
B_2'=\bigg(\bigcup_{E\cap F_i =\emptyset} E\times I\bigg) \cup 
\bigg( \bigcup_{\substack{E\cap F_i \neq \emptyset \\ b_1\notin E}} E_{(E\cap F_i)} \bigg), 
\end{align*}
where $E$ is a face of $Q$, and $E_{(E\cap F_i)}$ is the polytopal wedge of $E$ by 
considering $E$ as a simple polytope and $E\cap F_i$
as a facet of $E$. In general, if a face $E$ of a simple 
polytope $Q$ intersects a facet $F$ of $Q$, then $E$ is a face 
of $F_i$ or $E\cap F_i$ is a facet of $E$. 
Observe that the face structure of $B_2'$ 
is naturally inherited from the face structure of $Q_{(J')}$. 
In particular, neither $Q^+$ nor $Q^-$ is a face of $B_2'$. 

Next, we consider the following two possibilities: 
(i) $b_2\in V(F_i)\setminus \{b_1\}$, i.e., $\beta_2=2$, and (ii) $b_2\notin V(F_i)$.
If $b_2\in V(F_i)\setminus \{b_1\}$, we set $b_2'=b_2^-$ and 
$E_2'={E_2}_{(E_2\cap F_i)}$. 
If $b_2\notin V(F_i)$, we set $b_2' =b_2^+$ and $E_2'=E_2\times I$. 
Then, $b_2'$ has a neighborhood 
homeomorphic to $\RR^{\dim E_2+1}_{\geq 0}$ in $E_2'$, 
because $b_2$ has the appropriate neighborhood in $E_2$. 
Hence, we can define the second term $(B_2', E_2', b_2')$. 
The first two retraction sequences in Figure \ref{fig_ret_wedge_5-gon}
illustrate this case when $Q$ is a pentagon. 

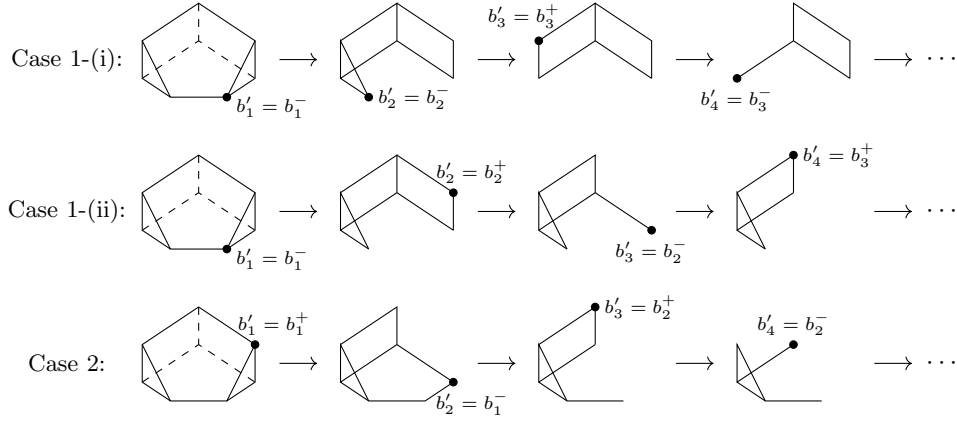
\begin{figure}
\begin{tikzpicture}[scale=0.25]

\begin{scope}[yshift=-460]
\node at (-4, 3) {\small{Case 2:}};
\draw[fill] (6,4) circle (0.2);
\node[above] at (7,4) {\scriptsize$b_1'=b_1^+$};
\draw (1.5,1)--(0,4)--(3,6)--(6,4)--(4.5,1)--cycle;
\draw (1.5,1)--(0,2)--(0,4);
\draw (4.5,1)--(6,2)--(6,4);
\draw[dashed] (0,2)--(3,4)--(6,2);
\draw[dashed] (3,6)--(3,4);
\draw[->] (7.3,3)--(9.3,3);

\begin{scope}[xshift=300]
\draw[fill] (6,2) circle (0.2);
\node[below] at (7,2) {\scriptsize$b_2'=b_1^-$};
\draw (1.5,1)--(0,2)--(3,4)--(6,2)--(4.5,1)--cycle;
\draw (3,4)--(3,6)--(0,4);
\draw (0,2)--(0,4)--(1.5,1);
\draw[->] (7.3,3)--(9.3,3);
\end{scope}

\begin{scope}[xshift=600]
\draw[fill] (3,6) circle (0.2);
\node[right] at (3,6) {\scriptsize$b_3'=b_2^+$};
\draw (4.5,1)--(1.5,1)--(0,2)--(3,4);
\draw (3,4)--(3,6)--(0,4);
\draw (0,2)--(0,4)--(1.5,1);
\draw[->] (7.3,3)--(9.3,3);
\end{scope}

\begin{scope}[xshift=900]
\draw[fill] (3,4) circle (0.2);
\node[above] at (3,4) {\scriptsize$b_4'=b_2^-$};
\draw (4.5,1)--(1.5,1)--(0,2)--(3,4);
\draw (0,2)--(0,4)--(1.5,1);
\draw[->] (7.3,3)--(9.3,3);
\node at (11,3) {$\cdots$};
\end{scope}

\end{scope}

\begin{scope}[yshift=-230]
\node at (-4, 3) {\small{Case 1-(ii):}};
\draw[fill] (4.5,1) circle (0.2);
\node[right] at (4.5,0.5) {\scriptsize$b_1'=b_1^-$};
\draw (1.5,1)--(0,4)--(3,6)--(6,4)--(4.5,1)--cycle;
\draw (1.5,1)--(0,2)--(0,4);
\draw (4.5,1)--(6,2)--(6,4);
\draw[dashed] (0,2)--(3,4)--(6,2);
\draw[dashed] (3,6)--(3,4);
\draw[->] (7.3,3)--(9.3,3);

\begin{scope}[xshift=300]
\draw[fill] (6,4) circle (0.2);
\node[above] at (7,4) {\scriptsize$b_2'=b_2^+$};
\draw (1.5,1)--(0,2)--(3,4)--(6,2);
\draw (3,4)--(3,6)--(0,4);
\draw (0,2)--(0,4)--(1.5,1);
\draw (6,2)--(6,4)--(3,6);
\draw[->] (7.3,3)--(9.3,3);
\end{scope}

\begin{scope}[xshift=600]
\draw[fill] (6,2) circle (0.2);
\node[below] at (6,2) {\scriptsize$b_3'=b_2^-$};
\draw (1.5,1)--(0,2)--(3,4)--(6,2);
\draw (3,4)--(3,6)--(0,4);
\draw (0,2)--(0,4)--(1.5,1);
\draw[->] (7.3,3)--(9.3,3);
\end{scope}

\begin{scope}[xshift=900]
\draw[fill] (3,6) circle (0.2);
\node[right] at (3,6) {\scriptsize$b_4'=b_3^+$};
\draw (1.5,1)--(0,2)--(3,4);
\draw (3,4)--(3,6)--(0,4);
\draw (0,2)--(0,4)--(1.5,1);
\draw[->] (7.3,3)--(9.3,3);
\node at (11,3) {$\cdots$};
\end{scope}
\end{scope}

\draw[fill] (4.5,1) circle (0.2);
\node at (-4, 3) {\small{Case 1-(i):}};
\node[right] at (4.5,0.5) {\scriptsize$b_1'=b_1^-$};
\draw (1.5,1)--(0,4)--(3,6)--(6,4)--(4.5,1)--cycle;
\draw (1.5,1)--(0,2)--(0,4);
\draw (4.5,1)--(6,2)--(6,4);
\draw[dashed] (0,2)--(3,4)--(6,2);
\draw[dashed] (3,6)--(3,4);
\draw[->] (7.3,3)--(9.3,3);

\begin{scope}[xshift=300]
\draw[fill] (1.5,1) circle (0.2);
\node[right] at (1.5,1) {\scriptsize$b_2'=b_2^-$};
\draw (1.5,1)--(0,2)--(3,4)--(6,2);
\draw (3,4)--(3,6)--(0,4);
\draw (0,2)--(0,4)--(1.5,1);
\draw (6,2)--(6,4)--(3,6);
\draw[->] (7.3,3)--(9.3,3);
\end{scope}

\begin{scope}[xshift=600]
\draw[fill] (0,4) circle (0.2);
\node[above] at (-.8,4.2) {\scriptsize$b_3'=b_3^+$};
\draw (0,2)--(3,4)--(6,2);
\draw (3,4)--(3,6)--(0,4);
\draw (0,2)--(0,4);
\draw (6,2)--(6,4)--(3,6);
\draw[->] (7.3,3)--(9.3,3);
\end{scope}

\begin{scope}[xshift=900]
\draw[fill] (0,2) circle (0.2);
\node[below] at (0,2) {\scriptsize$b_4'=b_3^-$};
\draw (0,2)--(3,4)--(6,2);
\draw (6,2)--(6,4)--(3,6)--(3,4);
\draw[->] (7.3,3)--(9.3,3);
\node at (11,3) {$\cdots$};
\end{scope}

\end{tikzpicture}
\caption{Three retraction sequences of the wedge of $5$-gon.}
\label{fig_ret_wedge_5-gon}
\end{figure}

\item[(Case 2)] Here we assume that $b_1\notin V(F_i)$. 
Then, we take $b_1'=b_1^+$ by \eqref{eq_Q(J)_seq_of_free_vertices}, 
which gives 
\begin{equation*}
B_2'=Q^- \cup \bigg(\bigcup_{E\cap F_i =\emptyset} E\times I\bigg) \cup 
\bigg(\bigcup_{\substack{E\cap F_i \neq \emptyset \\ b_1\notin E}} E_{(E\cap F_i)}\bigg),
\end{equation*}
where $E$ is a face of $Q$. 
Next, we take $b_2'=b_1^-$ and $E_2'=Q^-$ which is the unique maximal 
face of $Q_{(J')}$ containing $b_2'$. 
Now, $B_3'$ is naturally defined by deleting faces of $Q^-$ which contains 
$b_2'$ from $B_2'$. Observe that neither $Q^+$ nor $Q^-$ is a face
of $B_3'$, because a vertex in each of $Q^+$ and $Q^-$ has been removed.
See the third retraction sequence in Figure \ref{fig_ret_wedge_5-gon} for an 
example. 
\end{enumerate}
In both (Case1) and (Case2), one can see that 
$$\gcd\{p, |G_{E_1'}(b_1')|\}=\gcd\{p, |G_{E_2'}(b_2')|\}=1$$
by Proposition \ref{prop_sub_local_group},  
Lemma \ref{lem_loc_goup_isom_for_v_0} and
Lemma \ref{lem_loc_group_isom_for_v_pm}.

Finally, the remaining terms of the desired retraction sequence 
$\{(B^{\prime}_s, E^{\prime}_s, b_s^{\prime})\}_{s=1}^{2\ell-k}$ can be obtained 
by setting the vertices in the sequence \eqref{eq_Q(J)_seq_of_free_vertices} 
as the desired sequence of free vertices $\{b_s'\}_{s=1}^{2\ell-k}$. 
Indeed, it is enough to verify the following claims:
\begin{enumerate}
\item[(a)] Assume that $b_r\notin V(F_i)$, $b_s'=b_r^+$ and $b_{s+1}'=b_r^-$. 
Then, $b_s'$ and $b_{s+1}'$  are free vertices in $B_s'$ and $B_{s+1}'$, 
respectively. In particular, $E_s'$ and $E_{s+1}'$ are determined by 
$E_r\times I$ and $E_r\times \{0\}$, respectively. 
\item[(b)] Assume that $b_r\in V(F_i)$ and $b_s'=b_r^-$. 
Then, $b_s'$ is a free vertex in $B_s'$ with a unique maximal face 
$E_s'={E_r}_{(E_r\cap F_i)}$.
\item[(c)] $\gcd\{p, |G_{E_s'}(b_s')|\}=1$ for each $s=1, \dots, \ell'$. 
\end{enumerate}
The claims follow by repeating the arguments in (Case 1) and (Case 2) above. 
\end{proof}

Now, Theorem \ref{thm_no_p-torsion}, \ref{thm_no-torsion} and Lemma \ref{lem_ret_simj} 
concludes the following two theorems. 

\begin{theorem}
Let $X:=X(Q, \lambda)$ be a toric orbifold satisfying the assumption of 
Theorem \ref{thm_no_p-torsion}. Then
the cohomology $H^\ast(X_{(J)};\ZZ)$ has no $p$-torsion and 
$H_{odd}(X_{(J)}; \ZZ_p)$ is trivial for arbitrary $J=(j_1, \dots, j_m)\in \NN^m$.
\end{theorem}

\begin{theorem}\label{thm_X_satisfies_assump_then_X(J)_eq_formal}
Let $X:=X(Q, \lambda)$ be a toric orbifold satisfying the assumption of 
Theorem \ref{thm_no-torsion}. 
Then, $H^\ast(X_{(J)};\ZZ)$ is torsion free and $H_{odd}(X_{(J)}; \ZZ)$ vanishes 
for arbitrary $J=(j_1, \dots,  j_m)\in \NN^m$.
\end{theorem}

\section{Application to the cohomology ring}\label{sec_cohom_ring}
Motivated by the results of \cite{BFR},  a notion of \emph{weighted Stanley--Reisner ring}
was introduced  and used in \cite{BSS} to explicitly
compute the singular cohomology ring with integer coefficients of 
spaces identified as integrally equivariantly formal projective toric orbifolds. 
In this section, we briefly introduce their results and 
study the integral cohomology ring of $X_{(J)}$. 

\subsection{Cohomology ring of $X$}
In this subsection, we briefly summarize the theory of weighted Stanley--Reisner 
ring of a 
\emph{simple lattice polytope}\footnote{A simple polytope in $\RR^n$ whose vertices 
belong to the lattice $\ZZ^n \subset \RR^n$. } $Q$ 
which defines a projective normal toric variety,  see \cite[Section 5]{BSS}. 
The original definition of weighted 
Stanley--Reisner ring is based on a polytopal fan, but we translate the notation 
of \cite{BSS} to a simple lattice polytope which defines a polytopal fan. 

Let $\mathcal{F}(Q)=\{F_1, \dots, F_m\}$ be the set of facets of $Q$.
Since $Q$ is a lattice polytope, we may choose a primitive inward normal vector 
$\lambda_i$ of each facet $F_i$. Moreover, if $F_{i_1}\cap \dots \cap F_{i_k}=\emptyset$, 
then $\lambda_{i_1}, \dots, \lambda_{i_k}$ is linearly independent, because $Q$
is simple. Hence, the set of primitive inward normal vectors forms 
an $\mathcal{R}$-characteristic pair, say $(Q, \lambda)$. 

Next, for each vertex $v=F_{s_1} \cap\dots \cap F_{s_n} \in V(Q)$, 
we recall from Section \ref{sec_toric_orb_orb_lens}, (item (5) in page \pageref{item_(5)}), 
the $(n\times n)$-matrix associated to $v$, 
that is
$$\Lambda_{v} := 
\left[ \begin{array}{c|c|c} \lambda(F_{s_1})^t & \cdots & \lambda(F_{s_n})^t\end{array} \right].$$
 
For each vertex $v=F_{s_1}\cap \dots \cap F_{s_n}$, we define a vector 
$$z^v:=(z^v_1, \dots, z^v_m) \in \bigoplus_m \QQ[u_1, \dots, u_n],$$
by the following rule:
\begin{enumerate}\label{int_vectors}
\item[(i)] $z^v_j=0$ if $j\notin \{s_1, \dots, s_n\}$,
\item[(ii)] $\begin{bmatrix}
z^v_{s_1}\\ \vdots\\ z^v_{s_n}
\end{bmatrix} = \Lambda_{v}^{-1} \cdot \begin{bmatrix}u_1 \\ \vdots \\ u_n\end{bmatrix}$
\end{enumerate}
where the operation on the right hand side is the usual matrix multiplication. 
Next, we define a subset of $\ZZ[x_1, \dots, x_m]$ as follows:
$${\rm Int}[Q, \lambda]:=\{ f(x_1, \dots, x_m)  \mid f(z^v)\in \ZZ[u_1, \dots, u_n],
~\text{ for all } v\in V(Q)\}.$$

\begin{remark}\label{rmk_u_i-variables}
The variables $u_1, \dots, u_n$ stand for the basis of $H^2(BT^n;\ZZ)$, 
where $T^n$ is the $n$-dimensional torus acting on $X(Q, \lambda)$. 
Indeed, one may regard polynomial rings $\QQ[u_1, \dots, u_n]$ and  $\ZZ[u_1, \dots, u_n]$
as $H^\ast(BT^n;\QQ)$ and $H^\ast(BT^n;\ZZ)$, respectively. 
We refer to \cite[Section 5.2]{BSS}. 
\end{remark}

The next proposition highlights critical properties of ${\rm Int}[Q,\lambda]$.
\begin{proposition}
\begin{enumerate}
\item\label{prop_subring} The subset ${\rm Int}[Q, \lambda]$ is a subring of 
$\ZZ[x_1, \dots, x_m]$.
\item\label{prop_SR-ideal} The \emph{Stanley--Reisner ideal} 
$\mathcal{I}_Q :=\big\langle \prod_{j=1}^k x_{i_j} \mid F_{i_1} \cap \dots \cap F_{i_k} 
= \emptyset \big\rangle$
of the ring $\ZZ[x_1, \dots, x_m]$ for $Q$ is again an ideal of ${\rm Int}[Q, \lambda]$. 
\end{enumerate}
\end{proposition}
\begin{proof}
The first statement is almost obvious. 
Next, if a monomial $f(x_1, \dots x_m)=\prod_{j=1}^k x_{i_j}$ 
is an element of $\mathcal{I}_Q$, it follows from item (i) above 
that $f(z^v)=0$ for all $v\in V(Q)$. 
Hence, $\mc{I}_Q$ is not only a subset, 
but also an ideal of ${\rm Int}[Q, \lambda]$.
\end{proof}

\begin{definition}\cite[Section 5]{BSS}
The {\it weighted Stanley--Reisner ring} $w\mc{SR}[Q, \lambda]$ of an 
$\mc{R}$-characteristic pair $(Q, \lambda)$, associated to a simple lattice polytope, 
is the subring of the Stanley--Reisner ring $\mathcal{SR}[Q]$ of $Q$ defined by 
the quotient:
$$w\mc{SR}[Q, \lambda]:= {\rm Int}[Q, \lambda] / \mathcal{I}_Q.$$
\end{definition}

We remark that if a simple lattice polytope is Delzant, i.e., normal vectors associated to 
facets intersecting a vertex form a $\ZZ$-basis, then $\Lambda_{v}^{-1}$ has integer entries. 
Hence, $f(z^v)$ is polynomial with integer coefficients for all $v\in V(Q)$, which says 
that $w\mc{SR}[Q, \lambda]$ is the usual Stanley--Reisner ring of a simple lattice polytope. Hence, 
$w\SR[Q, \lambda]$ contains geometric data about $X(Q,\lambda)$, 
including singularities, in addition to the combinatorial information about $Q$. 

\begin{theorem}\cite[Theorem 5.3]{BSS}\label{thm_cohom_of_X}
Let $X(Q, \lambda)$ be a projective toric variety over a simple lattice polytope $Q$ 
with $H^{odd}(X(Q, \lambda);\ZZ)=0$. Then, there is an isomorphism 
$$H^\ast(X(Q, \lambda);\ZZ) \cong w\SR[Q, \lambda]/\mathcal{J},$$
where $\mc{J}$ is the ideal generated by the linear elements 
$ \sum_{i=1}^m \left< \lambda_i, \mbf{e}_j \right>x_i$ for $j=1, \dots, n,$
where $\mbf{e}_j$ denotes the $j$-th standard unit vector in $\ZZ^n$.
\end{theorem}

\subsection{The cohomology ring of $X_{(J)}$}
In this final section, we study the relationship between the cohomology ring of 
$X$ and  that of $X_{(J)}:=X(Q_{(J)}, \lambda_{(J)})$. 
The next theorem follows directly 
from Theorem \ref{thm_no-torsion} and Theorem \ref{thm_cohom_of_X}. 
\begin{theorem}\label{thm_cohom_of_X(J)}
Let $(Q, \lambda)$ be an $\mathcal{R}$-characteristic pair satisfying 
the hypothesis of Theorem \ref{thm_no-torsion}. Then, 
$$H^\ast(X_{(J)};\ZZ) \cong w\SR[Q_{(J)}, \lambda_{(J)}] / \mathcal{J}_{(J)}, $$
where $\mathcal{J}_{(J)}$ is the ideal generated by 
\begin{equation}\label{eq_ideal_I(J)}
\left\{\sum_{i=1}^m \langle \lambda_i, e_j\rangle  x_{i1}~\Big|~  j=1, \dots, n\right\}\cup 
 \left\{ x_{it}=x_{i1}\mid  t=2, \dots, j_i  \right\}.
 \end{equation}
\end{theorem}

 If $X$ is a smooth toric manifold, then the cohomology ring $H^\ast(X_{(J)};\ZZ)$
 can be dramatically simplified, because 
 $w\SR[Q_{(J)}, \lambda_{(J)}]=\SR[Q_{(J)}, \lambda_{(J)}]$, which 
 tells us that the ring is generated by degree $2$ elements 
 and the second part of \eqref{eq_ideal_I(J)} reduces the degree $2$ elements 
 to the same generators as $w\SR[Q, \lambda]$. See \cite[Section 4]{BBCG15} for details. 
However, in general, the cohomology ring $H^\ast(X;\ZZ)$ of a toric orbifold $X$ is not
generated by degree $2$ elements, hence the multiplication structure has plenty of 
divisibility because of singularities.

Recall that $X_{(J)}$ can be obtained from 
$X$ by a sequence of simplicial wedge constructions.
Consider now the polytopal wedge 
$Q_{(J')}:=Q_{(F_i)}$ of the original simple lattice polytope $Q$ for 
some facet $F_i$ of $Q$. 
    
We finish this paper by studying the vectors $\{z^{v^\epsilon} \mid v^\epsilon \in V(Q_{(J')}) \}$, 
with respect to the $\mathcal{R}$-characteristic pair $(Q_{(J')}, {\lambda_{(J')}})$.
Recall $V(Q_{(J')})$ from \eqref{eq_vertices_of_Q_tilde}.
Theorem \ref{thm_J-fied_int_cond} below tells us how to get
$$z^{v^{\epsilon}}=(z^{v^\epsilon}_0, z^{v^\epsilon}_1, \dots, z^{v^\epsilon}_m) \in 
\bigoplus_{m+1} \ZZ[u_0, u_1, \dots, u_n]$$
from $\{z^v \mid v\in V(Q)\}$,  for each $v\in V(Q)$ and $\epsilon=+$ or  $-$.

\begin{remark}
The polynomial ring $\ZZ[u_0, u_1, \dots, u_n]$ stands for $H^\ast(BT^{n+1};\ZZ)$,
where $T^{n+1}$ is the $(n+1)$-dimensional torus acting on $X_{(J')}$. 
The canonical embedding of $T^n$, the acting torus on $X(Q, \lambda)$, into 
the last $n$-coordinates of $T^{n+1}$ yields a canonical surjection
$\ZZ[u_0, u_1, \dots, u_n] \twoheadrightarrow  \ZZ[u_1, \dots, u_n]$. See
Remark \ref{rmk_u_i-variables}.
\end{remark}

According to \eqref{eq_wedge_of_K} and \eqref{eq_vertices_of_Q_tilde}, 
we have the following 3 types of vertices in $Q_{(J')}$; 
\begin{align}
	\label{eq_type1}&v^-_{i_r}=Q^+ \cap Q^- \cap \bigcap_{a=1}^{n-1}{F_{s_a}}_{(F_i \cap F_{s_a})}
		\text{ for some } v=F_i \cap  \bigcap_{a=1}^{n-1} F_{s_a} \in V(Q), \\
	\label{eq_type2}& v^+= Q^+ \cap \bigcap_{a=1}^n (F_{s_a} \times I)
	 	\text{ for some } v=\bigcap_{a=1}^n F_{s_a} \text{ and } v\notin V(F_i),  \\
	\label{eq_type3}&v^-= Q^- \cap \bigcap_{a=1}^n (F_{s_a} \times I)
	 	\text{ for some } v=\bigcap_{a=1}^n F_{s_a} \text{ and } v\notin V(F_i). 
\end{align}
We refer also to \eqref{eq_v_in_F_i_intersection_of_facets} and 
\eqref{eq_vertex_in_Q(J)_not_in_F_i} for the indexing of facets in the above three cases. 

\begin{theorem}\label{thm_J-fied_int_cond}
The weighted Stanley--Reisner ring of $(Q_{(J')}, \lambda_{(J')})$ is related to 
that of $(Q, \lambda)$ as follows. 
For each vertex of one of the three types \eqref{eq_type1}, 
\eqref{eq_type2} and \eqref{eq_type3} above, we have:
\begin{enumerate}
\item\label{item_(1)} $z^{v_{i_r}^-}= (u_0+z^{v_{i_r}}_i, z^{v_{i_r}}_1, \dots, z^{v_{i_r}}_m)$ for $r=1, \dots, k$;
\item\label{item_(2)} $z^{v^+}= (u_0, z^{v}_1, \dots, z^{v}_m)$;
\item\label{item_(3)}    $z^{v^-}=(0, z^{v^-}_1, \dots, z^{v^-}_m)$, where 
$$z^{v^-}_\ell=\begin{cases} 0 & \text{ if } \ell \notin \{s_1, \dots, s_n\} \\ 
\gamma_\ell u_0 + z^{v}_\ell & \text{ if } \ell \in \{s_1, \dots, s_n\} \end{cases} ~ \text{ and }~ 
\begin{bmatrix} \gamma_{1} \\ \vdots \\ \gamma_{n} \end{bmatrix} 
=\Lambda_{v}^{-1}\cdot \lambda(F_i)^t.$$
\end{enumerate}
\end{theorem}

\begin{proof}
The proof follows from the direct computation of the inverse of 
${\Lambda_{(J')}}_{v^\epsilon}$
for each $\epsilon=+$ and $-$. If $v$ is a vertex in $F_i$, i.e., $v=v_{i_r}$
for some $i_r\in \{i_1, \dots, i_k\}$, and 
$v_{i_r}=F_{i} \cap F_{s_1} \cap \dots \cap F_{s_{n-1}}$, 
then $\Lambda_{v}=
\left[ \begin{array}{c|c|c|c} \lambda(F_i)^t & \lambda(F_{s_1})^t &
\cdots & \lambda(F_{s_{n-1}})^t \end{array}\right]$ 
and 
$$
{\Lambda_{(J')}}_{v^-_{i_r}}=
	\left[\begin{array}{ccccc}1&-1& 0&\cdots& 0 \\ 
         0&  &   &  &    \\
          \vdots  &\lambda(F_i)^t    &  \lambda(F_{s_1})^t   
          &\cdots & \lambda(F_{s_{n-1}})^t  \\
          0&  & & &  \end{array} \right].$$
Its inverse is 
$${\Lambda_{(J')}}_{v^-_{i_r}}^{-1}
=\left[\begin{array}{c|ccc} 
1  & d_{1}&\cdots & d_{n}\\ \hline
0     &&&\\
\vdots&& \Lambda_{v}^{-1}& \\
0     && & \end{array}\right],$$
where $(d_1, \dots, d_n)$ is the first row of $\Lambda_{v}^{-1}$. 
Hence, by (i) and (ii) in page \pageref{int_vectors}, we conclude \eqref{item_(1)}. 

Next, we assume that $v=F_{s_1} \cap \dots \cap F_{s_n}$ with $i\notin \{s_1, \dots, s_n\}$. 
The square matrices corresponding to  
$v^+=Q^+ \cap \bigcap_{a=1}^n (F_{s_a} \times I)$ and 
$v^-=Q^- \cap \bigcap_{a=1}^n (F_{s_a} \times I)$ are 
\begin{align*}
{\Lambda_{(J')}}_{v^+}&=\left[\begin{array}{c|ccc}1          &   0       &   \cdots  &   0 \\\hline
                                               0&     &   & \\
                                            \vdots      &          &  \Lambda_{v} &  \\
                                            0&   &    &   \end{array}\right]=
\left[\begin{array}{cccc}
1& 0& \cdots &0  \\ 
0&&&\\
\vdots&  \lambda(F_{s_1})^t &\cdots&  \lambda(F_{s_n})^t \\
0&&&
\end{array}\right], \\
{\Lambda_{(J')}}_{v^-} &=
\left[ \begin{array}{c|ccc}
1 & 0 & \cdots & 0\\ \hline 
&&&\\
\lambda(F_i)^t & & \Lambda_v & \\
&&&
\end{array}
\right]=
\left[\begin{array}{cccc}
-1& 0& \cdots &0  \\ 
&&&\\
\lambda(F_i)^t&  \lambda(F_{s_1})^t &\cdots&  \lambda(F_{s_n})^t \\
&&&
\end{array}\right].                                                            
\end{align*}
respectively. Their inverses are 
\begin{equation}\label{eq_type_2_inverse}
{\Lambda_{(J')}}_{v^+}^{-1}=
	\left[\begin{array}{c|ccc}1          &   0       &   \cdots  &   0 \\ \hline
                                               0&     &   & \\
                                            \vdots      &          &  \Lambda_{v}^{-1}         &  \\
                                            0&   &    &   \end{array}\right]
\end{equation}
and 
\begin{equation}\label{eq_type_3_inverse}
{\Lambda_{(J')}}_{v^-}^{-1} = \left[ \begin{array}{c|ccc}
-1 & 0 & \cdots & 0\\ \hline
\gamma_1 & & &\\
\vdots & &\Lambda_v^{-1} &\\
\gamma_n & & &
\end{array}\right], \text{ where } 
\begin{bmatrix} \gamma_1 \\ \vdots \\ \gamma_n \end{bmatrix} 
        =\Lambda_{v}^{-1}\cdot \lambda(F_i)^t.
\end{equation}
Now, the results for \eqref{item_(2)} and \eqref{item_(3)} are straightforward from 
\eqref{eq_type_2_inverse} and \eqref{eq_type_3_inverse}, respectively. 
\end{proof}

%

\begin{thebibliography}{amsalpha}

\bibitem[AM69]{Ati}
M.~F. Atiyah and I.~G. Macdonald, \emph{Introduction to commutative algebra},
  Addison-Wesley Publishing Co., Reading, Mass.-London-Don Mills, Ont., 1969.
  \MR{0242802}

\bibitem[BBCG]{BBCG18}
A.~Bahri, M.~Bendersky, F.~R. Cohen, and S.~Gitler, \emph{A generalization of
  the davis--januszkiewicz construction and applications to toric manifolds and
  iterated polyhedral products}, To appear in: Perspectives in Lie Theory,
  Springer. Online at: http://arxiv.org/abs/1311.4256.

\bibitem[BBCG15]{BBCG15}
\bysame, \emph{Operations on polyhedral products and a new topological
  construction of infinite families of toric manifolds}, Homology Homotopy
  Appl. \textbf{17} (2015), no.~2, 137--160. \MR{3426378}

\bibitem[BFNR13]{BFNR}
Anthony Bahri, Matthias Franz, Dietrich Notbohm, and Nigel Ray, \emph{The
  classification of weighted projective spaces}, Fund. Math. \textbf{220}
  (2013), no.~3, 217--226. \MR{3040671}

\bibitem[BFR09]{BFR}
Anthony Bahri, Matthias Franz, and Nigel Ray, \emph{The equivariant cohomology
  ring of weighted projective space}, Math. Proc. Cambridge Philos. Soc.
  \textbf{146} (2009), no.~2, 395--405. \MR{2475973 (2010a:57054)}

\bibitem[BNSS]{BNSS}
Anthony Bahri, Dietrich Notbohm, Soumen Sarkar, and Jongbaek Song, \emph{On
  integral cohomology of certain orbifolds}, arXiv:1711.01748.

\bibitem[BP15]{BP-book}
Victor~M. Buchstaber and Taras~E. Panov, \emph{Toric topology}, Mathematical
  Surveys and Monographs, vol. 204, American Mathematical Society, Providence,
  RI, 2015. \MR{3363157}

\bibitem[BSS17]{BSS}
Anthony Bahri, Soumen Sarkar, and Jongbaek Song, \emph{On the integral
  cohomology ring of toric orbifolds and singular toric varieties}, Algebr.
  Geom. Topol. \textbf{17} (2017), no.~6, 3779--3810. \MR{3709660}

\bibitem[CLS11]{CLS}
David~A. Cox, John~B. Little, and Henry~K. Schenck, \emph{Toric varieties},
  Graduate Studies in Mathematics, vol. 124, American Mathematical Society,
  Providence, RI, 2011. \MR{2810322 (2012g:14094)}

\bibitem[CP16]{CP}
Suyoung Choi and Hanchul Park, \emph{Wedge operations and torus symmetries},
  Tohoku Math. J. (2) \textbf{68} (2016), no.~1, 91--138. \MR{3476138}

\bibitem[CP17]{CP2nd}
\bysame, \emph{Wedge {O}perations and {T}orus {S}ymmetries {II}}, Canad. J.
  Math. \textbf{69} (2017), no.~4, 767--789. \MR{3679694}

\bibitem[DJ91]{DJ}
Michael~W. Davis and Tadeusz Januszkiewicz, \emph{Convex polytopes, {C}oxeter
  orbifolds and torus actions}, Duke Math. J. \textbf{62} (1991), no.~2,
  417--451. \MR{1104531 (92i:52012)}

\bibitem[Ewa86]{Ew}
G\"unter Ewald, \emph{Spherical complexes and nonprojective toric varieties},
  Discrete Comput. Geom. \textbf{1} (1986), no.~2, 115--122. \MR{834053}

\bibitem[Fis92]{Fis}
Stephan Fischli, \emph{On toric varieties}, Ph.D. thesis, Universit{\"a}t Bern
  (1992).

\bibitem[FP07]{FP}
Matthias Franz and Volker Puppe, \emph{Exact cohomology sequences with integral
  coefficients for torus actions}, Transform. Groups \textbf{12} (2007), no.~1,
  65--76. \MR{2308029}

\bibitem[Ful93]{Ful}
William Fulton, \emph{Introduction to toric varieties}, Annals of Mathematics
  Studies, vol. 131, Princeton University Press, Princeton, NJ, 1993, The
  William H. Roever Lectures in Geometry. \MR{1234037 (94g:14028)}

\bibitem[GKM98]{GKM}
Mark Goresky, Robert Kottwitz, and Robert MacPherson, \emph{Equivariant
  cohomology, {K}oszul duality, and the localization theorem}, Invent. Math.
  \textbf{131} (1998), no.~1, 25--83. \MR{1489894}

\bibitem[Ill78]{Il}
S\"{o}ren Illman, \emph{Smooth equivariant triangulations of {$G$}-manifolds
  for {$G$} a finite group}, Math. Ann. \textbf{233} (1978), no.~3, 199--220.
  \MR{0500993 (58 \#18474)}

\bibitem[Jor98]{Jor}
Arno Jordan, \emph{Homology and cohomology of toric varieties}, Ph.D. thesis,
  University of Konstanz (1998).

\bibitem[Kaw73]{Ka}
Tetsuro Kawasaki, \emph{Cohomology of twisted projective spaces and lens
  complexes}, Math. Ann. \textbf{206} (1973), 243--248. \MR{0339247}

\bibitem[KMZ17]{KMZ}
Hideya Kuwata, Mikiya Masuda, and Haozhi Zeng, \emph{Torsion in the cohomology
  of torus orbifolds}, Chin. Ann. Math. Ser. B \textbf{38} (2017), no.~6,
  1247--1268. \MR{3721698}

\bibitem[LdM89]{LDM}
Santiago L\'opez~de Medrano, \emph{Topology of the intersection of quadrics in
  {${\bf R}^n$}}, Algebraic topology ({A}rcata, {CA}, 1986), Lecture Notes in
  Math., vol. 1370, Springer, Berlin, 1989, pp.~280--292. \MR{1000384}

\bibitem[PB80]{PB}
J.~Scott Provan and Louis~J. Billera, \emph{Decompositions of simplicial
  complexes related to diameters of convex polyhedra}, Math. Oper. Res.
  \textbf{5} (1980), no.~4, 576--594. \MR{593648 (82c:52010)}

\bibitem[PS10]{PS}
Mainak Poddar and Soumen Sarkar, \emph{On quasitoric orbifolds}, Osaka J. Math.
  \textbf{47} (2010), no.~4, 1055--1076. \MR{2791564 (2012e:57058)}

\bibitem[Zie95]{Zie}
G{\"u}nter~M. Ziegler, \emph{Lectures on polytopes}, Graduate Texts in
  Mathematics, vol. 152, Springer-Verlag, New York, 1995. \MR{1311028
  (96a:52011)}

\end{thebibliography}

\end{document}